\documentclass[12pt,twoside]{article} %
\usepackage{latexsym,amsmath,amssymb,amsthm,amsfonts} 
\usepackage{hyperref}

\usepackage{mathrsfs} 

\usepackage{color}

\definecolor{blue}{rgb}{0.00,0.00,1.00}
\definecolor{red}{rgb}{1.00,0.00,0.00}

 \renewcommand\baselinestretch{1.2}
 \numberwithin{equation}{section} \allowdisplaybreaks

 \newtheorem{theorem}{Theorem}[section]
 \newtheorem{lemma}{Lemma}[section]
 \newtheorem{definition}{Definition}[section]
 
 \newtheorem{proposition}[lemma]{Proposition}
 \newtheorem{remark}[lemma]{Remark} 

\newcommand{\R}{{\mathbb R}}

 \def\be#1\ee{\begin{equation}#1\end{equation}}
 \def\bma#1\ema{{\allowdisplaybreaks\begin{align}#1\end{align}}}
  \def\bman#1\eman{{\allowdisplaybreaks\begin{align*}#1\end{align*} }}
        \def\nnm{\notag}
 \def\bgr#1\egr{{\allowdisplaybreaks\begin{gather}#1\end{gather}}}
 \def\bgrn#1\egrn{{\allowdisplaybreaks\begin{gather}#1\end{gather}}}
 \def\ef#1{(\ref{#1})}
 
 \def\eqref#1{$(\ref{#1})$}

   \def\r0{r_b}
\allowdisplaybreaks

\begin{document}

\title{\bf
Local existence of solution to free boundary value problem for compressible Navier-Stokes equations}

\author{Jian Liu
 \\
{\small\it Department of Mathematics, Capital Normal University}\\
{\small\it Beijing 100048, P.R. China. E-mail: liujian.maths@gmail.com}
 }

\thispagestyle{empty}

\pagestyle{myheadings}
 \markboth{\small{J. Liu}}
  {\small{FBVP for CNS}}

\date{}

\maketitle
\begin{abstract}
\noindent \textbf{Abstract} \ \ This paper is concerned with the free boundary value problem for multi-dimensional Navier-Stokes equations with
density-dependent viscosity where the flow density vanishes
continuously across the free boundary. A local (in time) existence
of weak solution is established, in particular, the density is
positive and the solution is regular away from the free boundary.

\noindent{\small\bf Key words} \ \ Navier-Stokes equations; free boundary value problem; local existence

\noindent{\small\bf 2000 MR Subject
Classification} 35Q35 76D03
\end{abstract}


\section{Introduction}
\label{introduction}
The compressible Navier-Stokes equations (CNS) with
density-dependent viscosity coefficients are taken into granted recently. The
prototype is the model of viscous Saint-Venat system used in geophysical flow [13] to simulate the motion of the surface in shallow water, of which the mathematical derivation is also made recently based on the motion of three dimensional incompressible viscous fluids on shallow region with free surface condition on the top and Navier type boundary condition at bottom of finite depth [6, 10].

In the present paper, we consider the general isentropic compressible Navier-Stokes equations with density-dependent viscosity coefficients in $\mathbb{R}^N$, $N=2,3$, can be written for $t>0$ as
 \be
 \left\{\begin{aligned}
 &\rho_t+\textrm{div}(\rho \mathbf{U})=0,\\
 &(\rho\mathbf{U})_t
 +\textrm{div}(\rho\mathbf{U}\otimes\mathbf{U})
 -\textrm{div}(\mu(\rho)\mathbb{D}(\mathbf{U}))
 -\nabla(\lambda(\rho)\textrm{div}\mathbf{U})
 +\nabla P(\rho)=0,          \label{1.2}
\end{aligned}\right.
 \ee
where $\rho(\mathbf{x},t)$, $\mathbf{U}(\mathbf{x},t)$ and
$P(\rho)={\rho}^{\gamma}(\gamma >1)$ stand for the fluid density,
velocity and pressure, respectively, $
\mathbb{D}(\mathbf{U})=\frac{1}{2}(\nabla\mathbf{U}+
\nabla\mathbf{U}^T) $ is the stress tensor, and $\mu(\rho)$ and
$\lambda(\rho)$ are the Lam\'{e} viscosity coefficients satisfying
$\mu(\rho)\ge0$ and $\mu(\rho)+N\lambda(\rho)\geq 0$ for $\rho\ge0$.
Note here that the case $\gamma=2$ and $\theta=1$ in \eqref{1.2}
corresponds to the viscous Saint-Venat system.
\par

One of mathematical difficulties to investigate the existence and
dynamics of solutions to \eqref{1.2} is that the viscosity
coefficients are density-dependent which leads to strong degeneracy
in the appearance of vacuum [5]. Thus, it is natural and interesting
to investigate the influence of vacuum state on the existence and
dynamics of global solutions to \eqref{1.2}. One of the prototype
problems is the time-evolution of the compressible viscous flow of
finite mass expanding into infinite vacuum. This corresponds to free
boundary value problem (FBVP) for the compressible Navier-Stokes
equations \eqref{1.2} for general initial data and variant boundary
conditions imposed on the free surface. The study is fundamental
issue of fluid mechanics and has attracted lots of research
interests [11, 17]. These free boundary problems have been studied
with rather abundant results concerned with the existence and
dynamics of global solution for CNS \eqref{1.2} in 1D, refer to [4,
9, 14, 15, 19, 20] and references therein. As for related phenomena
of vacuum vanishing and dynamics of free boundary, the reader can
refer to [8, 9].

The free boundary value problem for \eqref{1.2} with stress free
boundary condition has been investigated in [7], where global
existence of spherically symmetric weak solution is shown, in
particular, the dynamics behaviors and the Lagrangian properties are
also established therein. Chen-Zhang [3] proved the local solutions
of \eqref{1.2} with spherically symmetric initial data between a
solid core and a free boundary connected to a surrounding vacuum
state. Under certain assumptions that are imposed on the spherically
symmetric initial data, which between a solid core and a free
boundary, Chen-Fang-Zhang established the global existence,
uniqueness and continuous dependence on initial data of a weak
solution in [2]. Wei-Zhang-Fang obtained the global existence and
uniqueness of the  spherically symmetric weak solution in [18] with
the symmetric center excluded.

In the present paper, we consider the free boundary value problem
for multi-dimensional CNS \eqref{1.2} in the case that where the
fluid density connects with vacuum continuously. We show that a
spherically symmetric weak solution, with the symmetric center
included, exists locally in time, in particular the density is
positive away from the free boundary but vanishes across the initial
interface separating fluids and vacuum,  and the free surface moves
as particle pathes in radial direction.
To this end, we need to employ the basic energy and the modified Bresch-Desjardins (BD) [1]
entropy to establish the expected boundary regularities of spherically
symmetric solutions in Lagrangian coordinates so as to control
the finite speed motion of free boundary within finite time.
Then, in terms of the original equations instead of the spherically
symmetric form, we are able to apply the higher order energy estimates
to establish the necessary interior regularities of solutions away from the
free boundary but with the symmetry center included. Then, the combination of
both boundary estimates and interior estimates and the above leads to the desired
local existence and uniqueness results of solutions.

The rest of this paper is as follows. In Section 2, we state the
main results  of this paper. In Sections 3-5 we establish boundary
regularity and interior regularity, with which we can prove the
existence and uniqueness in Section 6.


\setcounter{equation}{0}
\section{Main results}
\label{FBVP}

For simplicity, the viscosity terms are assumed to satisfy
$\mu(\rho)=\rho^{\theta}$, $\lambda(\rho)=\rho \mu'(\rho)-\mu(\rho)=(\theta-1)\rho^\theta$ and
$\mathbb{D}(\mathbf{U})=\frac{1}{2}(\nabla\mathbf{U}+
\nabla\mathbf{U}^T)$ in \ef{1.2}. The
pressure is assumed to be $P(\rho)=\rho^\gamma$. In this situation,
 \eqref{1.2} become
 \be
\left\{\begin{aligned}
&\rho_t+\textrm{div}(\rho \mathbf{U})=0,\\
&(\rho\mathbf{U})_t+\textrm{div}(\rho\mathbf{U}\otimes\mathbf{U})
-\textrm{div}(\rho^\theta\mathbb{D}(\mathbf{U}))-
(\theta-1)\nabla(\rho^\theta\textrm{div}\mathbf{U})
+\nabla\rho^\gamma=0.  \label{2.1o}
\end{aligned}\right.
 \ee
Consider a spherically symmetric solution $(\rho,\mathbf{U})$ to \eqref{2.1o} in $\R^N$ so that
\be
 \rho(\mathbf{x},t)=\rho(r,t),\quad
 \rho\mathbf{U}(\mathbf{x},t)=\rho u(r,t)\frac{\mathbf{x}}{r},
 \quad r=|\mathbf{x}|,
 \quad \mathbf{x}  \in\R^N,                      \label{2.4}
 \ee
and \eqref{2.1o} are changed to
\be
 \left\{\begin{aligned}
 &(r^{N-1}\rho)_t+(r^{N-1}\rho u)_r=0,                          \label{2.1}
 \\
 &(r^{N-1}\rho u)_t+(r^{N-1}\rho u^2)_r +r^{N-1}(\rho^{\gamma})_r\\
  &- r^{N-1}(\theta\rho^{\theta}(u_r+\mbox{$\frac{N-1}{r}$}u))_r
  + (N-1) r^{N-2} (\rho^{\theta})_r u  =0,
\end{aligned}\right.
 \ee
for  $(r,t)\in \Omega_T$ with
 \be
 \Omega_T=\{(r,t)|\, 0\le r\le a(t),\ 0\le t\le T \}. \label{omega-t}
 \ee
The initial data is taken as
 \be
 \mbox{$(\rho_0,\mathbf{U}_0)(\mathbf{x}) =
 (\rho_0(r), u_0(r)\frac{\mathbf{x}}{r})$},
 \quad r\in[0,a_0].    \label{2.1a}
 \ee
At the center of symmetry we impose the Dirichlet boundary condition
\begin{eqnarray}
  u(0,t)=0,           \label{2.1b}
\end{eqnarray}
and the free surface $\partial\Omega_t$ moves in radial direction
along the ``particle path" $r=a(t)$ with the stress-free boundary
condition
\begin{eqnarray}
\rho(a(t),t)=0,\quad t> 0, \label{2.1d}
\end{eqnarray}
where $a(t)$ is the free boundary defined by
\begin{eqnarray}
a'(t)=u(a(t),t), t>0,\quad  a(0)=a_0. \label{2.1f} 
\end{eqnarray}

First, we define a weak solution to the
FBVP~\eqref{2.1o}-\eqref{2.1f} as follows.

\begin{definition}\label{defi}
$(\rho,\mathbf{U},a)$ with $\rho\geq 0$ a.e. is said to be a weak
solution to the free surface problem~\eqref{2.1o}-\eqref{2.1f}  on
$\Omega_t\times[0,T]$, provided that it holds
 \bgr
 0\le\rho\in L^\infty(0,T; L^1(\Omega_t)\cap L^\gamma(\Omega_t)),\ \
         \sqrt{\rho}\,\mathbf{U}\in L^\infty(0,T;L^2(\Omega_t)),\nnm
 \\
\rho^{\frac{\theta}{2}}\,\nabla\mathbf{U}\in L^2(0,T;L^{2}(\Omega_t)),\ \
 a(t)\in C^0([0,T]),\nnm
 \egr
and the equations are satisfied in the sense of distributions.
Namely, it holds for any $t_2> t_1\ge 0$ and $\phi\in
C^1(\bar{\Omega}_t\times[0,T])$ that
\be
 \int_{\Omega_t} \rho\phi d\mathbf{x}|_{t=t_1}^{t_2}
 =
 \int_{t_1}^{t_2}\int_{\Omega_t}
 (\rho \phi_t+\rho\mathbf{U}\cdot\nabla\phi)d\mathbf{x}dt  \label{mass}
 \ee
and for $\psi=(\psi^1,\psi^2,\cdots\psi^N)\in
C^1(\bar{\Omega}_t\times[0,T])$ satisfying $\psi(\mathbf{x},t)=0$ on
$\partial\Omega_t$ and $\psi(\mathbf{x},T)=0$ that
\bma
 & \int_{\Omega_t}\mathbf{m_0}\cdot\psi(\mathbf{x},0)d\mathbf{x}
   +\int_0^T\int_{\Omega_t} \rho^\gamma {\rm div}  \psi d\mathbf{x}dt
     -(\theta-1)\int_0^T\int_{\Omega_t}\rho^\theta
      {\rm div}\mathbf{U}{\rm div}\psi d\mathbf{x}dt
  \nonumber\\[2mm]
 &-\int_0^T\int_{\Omega_t}\rho^\theta\nabla\mathbf{U}:\nabla\psi d\mathbf{x}dt+
  \int_0^T\int_{\Omega_t}[\rho\mathbf{U}
     \cdot\partial_t\psi
  +\sqrt{\rho}\,\mathbf{U}\otimes\sqrt{\rho}\,\mathbf{U}:\nabla\psi]
    d\mathbf{x}dt=0 \label{momentum}
 \ema
where $\mathbf{m_0}=m_0\frac{\mathbf{x}}{r}$. The free boundary
condition \ef{2.1d} is satisfied in the sense of continuity.
\end{definition}

\bigskip
\noindent\textbf{Notations}:\ Throughout this paper, $C$ and $c$
denote generic positive constants, $C_{f,g}>0$ denotes a generic
constant which may depend on the sub-index $f$ and $g$, and $C_T>0$
a generic constant dependent of $T>0$.

Before stating the main result, we need assume the initial data \ef{2.1a}
satisfies for $0<r_0<r_2<r_1<r_1^+<a_0$ that
 \be
 \left\{\begin{aligned}
 & \mbox{$\int_0^{a_0} $}r^{N-1}\rho_0(r)dr=1,
 \\
 & \rho_*(a_0-r)^\sigma\le \rho_0(r)\le \rho^*(a_0-r)^\sigma,\
    r\in[0,a_0],
 \\
 &(\rho_0^{-\frac{1}{2}}(\rho_0^{\theta}u_{0r})_r,
  \rho_0^{-1}\rho_{0r}u_0)\in L^2([r_2,a_0]),\quad
  u_0\in H^1([r_2,a_0]),
 \\
 &\rho_0u_0^{4m}\in L^1([r_2,a_0]),
 \\
 & (\rho_0,\mathbf{U}_0)\in H^{3}([0,r_1^+]),
  \quad  (\sqrt{\rho_0}\,\mathbf{U}_0,\rho^{\gamma/2}_0) \in L^2([0,a_0]),
\end{aligned}
\right.                \label{2.11a2}
\ee
where
$\rho_*$ and $\rho^*$ are positive
constants.

Meanwhile we list some assumptions on the constants $(\gamma, \theta, \beta, m)$ with $\beta=\frac{\sigma}{1+\sigma}$.\\
($A_1$) Let $\gamma$, $\theta$ satisfy
\be
\mbox{$\frac{N-1}{N}$}<\theta<\gamma,\quad \gamma>1.
\ee
($A_2$) Let $\beta$ satisfying
\be
\frac{1}{2\gamma}<\beta<\min\{\frac{1}{2\theta},\frac{1}{1+\theta}\},\quad \beta(\theta-1)<\frac13.
\ee
($A_3$) $m>0$ is a integer satisfying
\be
m>\max\{\frac{1}{1+\beta\theta-\beta},\frac{1}{4-4\beta}\}.\label{m}
\ee
Under the above assumptions, we have the following existence result.

\begin{theorem}  \label{thm.existence}
Let $N=2,3$, $\gamma>1$. Assume that \eqref{2.11a2} and $A_1 \sim A_3$ hold, Then,
there exist a time $T_*>0$ and $\rho_\pm>0$
dependent of initial data, so that the
FBVP~\eqref{2.1o}-\eqref{2.1f} has a unique spherically symmetric
weak solution for $t\in[0,T_*]$
\be
 \mbox{$(\rho,\rho\mathbf{U},a)(\mathbf{x},t) =
 (\rho(r,t),\rho u(r,t)\frac{\mathbf{x}}{r},a(t))$},
    \quad r=|\mathbf{x}|,        \nnm
 \ee
in the sense of Definition~\ref{defi} for any $T\in(0,T_*]$
satisfying that
\be
  \int_{0}^{a(t)}
   r^{N-1}\rho (r,t)dr
  =\int_{0}^{a_0} r^{N-1} \rho_0(r)dr,      \label{th2.1a9}
\ee
\be
  c_0\le a(t)\le 2a_0, \ \ t\in [0,T_*], \quad
 \|a\|_{H^{2}([0,T_*])}\le  C,   \label{th2.1a}
\ee
\be
   (\rho,\mathbf{U})\in C^0(\Omega_t\times[0,T_*]),  \ \
  \|\mathbf{U} \|_{W^{1,\infty}(\Omega_t\times[0,T_*])}  \le C,   \label{th2.1a2}
\ee
\be
 \rho_-(a(t)-r)^\sigma
     \le \rho(r,t)\le \rho_+(a(t)-r)^\sigma,
        \quad (r,t)\in[0,a(t)]\times[0,T_*], \label{th2.1a3}
\ee
\be
  \sup_{t\in [0,T_*]}\int_{\Omega_t}
   (\rho^{\gamma}+|\sqrt{\rho}\,\mathbf{U}|^2 )
   (\mathbf{x},t)d\mathbf{x}
  +\int_0^T\!\!\int_{\Omega_t}\rho^{\theta}|\nabla\mathbf{U}|^2
  d\mathbf{x}dt
 \le C,                                     \label{th2.1b}
\ee
\be
 \sup_{t\in [0,T_*]}\|(\rho,\mathbf{U})(t)\|_{H^3(\Omega^{in}_t)}
  +\int_0^{T_*} (
     \|\rho(t)\|^2_{H^3(\Omega^{in}_t)}
   +\|\nabla\mathbf{U}(t)\|^2_{H^3(\Omega^{in}_t)})dt \le C, \label{th2.1ba}
\ee
\be
\sup_{t\in [0,T_*]}\int_{r_{x_2}(t)}^{a(t)}\rho r^{N-1}(u^{2k}+u_t^2)dr
+\int_0^{T_*}\int_{r_{x_2}(t)}^{a(t)}\rho^{\theta}r^{N-1}(u^{2k-2}u_r^2
+u_{rt}^2+r^{-2}u_t^2)drdt\leq C, \label{th2.1bc}
\ee
where $\Omega_t=\{0\le |\mathbf{x}|\le a(t)\}$,
$\Omega^{in}_t=\{0\le |\mathbf{x}|\le r_{x_1}(t)\}$, $r_{x_i}(t)$ is
the particle path with $r_{x_i}(0)=r_i\, (i=1, 2)$ and $1\leq k\leq
2m$ is a integer, and $C>0$ is a constant.
\end{theorem}

\begin{remark}
Theorem~\ref{thm.existence} yields the local existence of
spherically symmetric weak solutions for two/three dimensional
compressible Navier-Stokes equation with fluid density connecting
with vacuum continuously. In particular, it applies to the viscous
Saint-Venant model for shallow water (which is \eqref{1.2} with
$N=2,\,\mu(\rho)=\rho,\, \lambda(\rho)=0$, and $\gamma=2$).
\end{remark}

\section{Basic energy estimates}
\label{construction}
The proof of Theorem~\ref{thm.existence} consists of the
construction of approximate solutions, the basic a-priori estimates,
and compactness arguments. We establish the a-priori estimates for
any solution $(\rho,u,a)$ to FBVP~\eqref{3.2}-\eqref{3.3} in this
section.

Let us introduce the Lagrangian coordinates transform
\be
 x(r,t)=\int_{0}^{\,r}\rho y^{N-1}dy
       =1-\int^{a(t)}_{\,r}\rho y^{N-1}dy,\quad  \tau=t, \label{lagrange}
 \ee
which translates the domain $[0,T]\times[0,a(t)]$ into
$[0,T]\times[0,1]$ and satisfies
 \be
 \frac{\partial x}{\partial r}=\rho r^{N-1},\quad
 \frac{\partial x}{\partial t}=-\rho ur^{N-1},\quad
 \frac{\partial \tau}{\partial r}=0, \quad
 \frac{\partial \tau}{\partial t}=1,               \label{lagrange.b}
\ee
and
 \be
 r^N(x,\tau)
  =N\int_0^x\frac1\rho(y,\tau)dy
  =a(t)^N-N\int_x^{1}\frac1\rho(y,\tau)dy,\quad
   \frac{\partial r}{\partial \tau}=u.               \label{3.1}
 \ee
In terms of \ef{lagrange}--\ef{3.1}, the
free boundary value problem \ef{2.1o}--\eqref{2.1f} is changed to
\be
\left\{\begin{aligned}
 &\rho_\tau+\rho^2(r^{N-1} u)_x=0,
 \\[2mm]
 & r^{1-N}u_\tau + (\rho^\gamma- \theta\rho^{\theta+1}(r^{N-1}u)_x)_x
                  + \mbox{$\frac{N-1}{r}$}(\rho^{\theta})_xu  = 0,
\end{aligned} \right.                                 \label{3.2}
\ee
for $(x,\tau)\in[0,1]\times[0,T]$, with the initial data and
boundary conditions given by
\begin{gather}
(\rho,u)(x,0)=(\rho_0,u_0)(x), \quad x\in[0,1],  \label{3.3a}
\\[2mm]
 u(0,\tau)=0, \quad
 \rho(1,\tau)=0,\quad \tau\in[0,T], \label{3.3}
\end{gather}
where $r=r(x,\tau)$ is defined by
\be
\frac{d }{d\tau}r(x,\tau)=u(x,\tau),
 \quad x\in[0,1],\ \tau\in[0,T],   \label{3.3e}
\ee
and the fixed boundary $x=1$ corresponds to the free boundary
$a(\tau)=r(1,\tau)$ in Eulerian form determined by
 \be
 \frac{d}{d\tau}a(\tau) = u(1,\tau),\ \tau\in[0,T],\ a(0)=a_0. \label{3.3b}
 \ee
Note that in Lagrange coordinates the condition \ef{2.11a2} is
equivalent to
 \be
 \left\{\begin{aligned}
 & \rho_*(1-x)^\beta\le \rho_0(x)\le \rho^*(1-x)^\beta,\
    x\in[0,1],
 \\
 &(\rho_0^{1+\theta}r^{N-1}u_{0x})_x\in L^2([x_2,1]),\quad
  \rho_0^{1/2}r^{N-1}u_0\in H^1([x_2,1]),
 \\
 &u_0^{4m}\in L^1([x_2,1]),
 \\
 & (\rho_0,u_0)\in H^{3}([0,x_1^+]),\  (u_0,\rho_0^{(\gamma-1)/2})\in L^2([0,1]),
\end{aligned}
\right.     \label{2.11a2a}
\ee
where
$0<x_2=\int_0^{r_2}r^{N-1}\rho_0(r)dr<x_1^+=\int_0^{r_1^+}r^{N-1}\rho_0(r)dr$.

First, making use of similar arguments as [7] with modifications, we can establish the following Lemmas \ref{lm3.1}--Lemmas \ref{lm3.3}, which we omit the details.
\begin{lemma}\label{lm3.1}
Let $\gamma>1$, $T>0$, and $(\rho,u,a)$ with $\rho>0$ be the
solution to the FBVP~\eqref{3.2}-\eqref{3.3} for $\tau\in[0,T]$.
Then, it holds
\begin{align}
 & \int_0^1
     (\mbox{$\frac{1}{2}$}u^2
      + \mbox{$\frac{1}{\gamma-1}$}\rho^{\gamma-1})dx
 +[1-N(1-\theta)](N-1)\int_0^\tau\int_0^1\rho^{\theta-1}\frac{u^2}{r^2}dxds
 \nnm\\
 &\,+[1-N(1-\theta)]\int_0^\tau\int_0^1
 \rho^{1+\theta}(r^{N-1}u_x)^2 dxds=E_0,\quad \tau\in [0,T],\label{3.4}
\end{align}
\end{lemma}
where
\be
  E_0=:\int_0^1(\mbox{$\frac{1}{2}$}u_0^2
      + \mbox{$\frac{1}{\gamma-1}$}\rho_0^{\gamma-1})dx. \nnm
\ee

\begin{lemma}
\label{lm3.2}
Under the same assumptions as Lemma~\ref{lm3.1}, it holds
\begin{align}
&\int_0^1
     (\mbox{$\frac{1}{2}$}u^2
      + \mbox{$\frac{1}{\gamma-1}$}\rho^{\gamma-1})dx
 +(\theta-1+\mbox{$\frac{1}{N}$})\int_0^\tau\int_0^1 \rho^{\theta+1}[(r^{N-1}u)_x]^2dxds
\nnm\\
   &+\mbox{$\frac{N-1}{N}$}\int_0^\tau\int_0^1
     \rho^{\theta+1}
       (r^{N-1}u_x-\mbox{$\frac{u}{r\rho}$})^2dxds=E_0, \quad \tau\in [0,T].  \label{3.40}
\end{align}
\end{lemma}

\begin{lemma}
\label{lm3.3}
Under the same assumptions as Lemma~\ref{lm3.1}, it holds
\bgr
 E_0^{-\frac{1}{N(\gamma-1)}}
  x^{\frac{\gamma}{N(\gamma-1)}}
 \le r(x,\tau)\le a(\tau),
 \quad (x,\tau)\in [0,1]\times[0,T],      \label{3.23a}
\\
 E_0^{-\frac{1}{\gamma-1}}
 (x_2-x_1) ^{\frac\gamma{\gamma-1}}
 \le
 r^N(x_2,\tau)- r^N(x_1,\tau),
 \quad 0\le x_1<x_2\le 1, \ \tau \in [0,T].   \label{3.23b}
 \egr
In particular, it holds for $x=1$ that
\begin{gather}
 E_0^{-\frac{1}{N(\gamma-1)}}
  \le a(\tau)\equiv r(1,\tau), \quad \tau \in [0,T].    \label{3.50}
\end{gather}
\end{lemma}
Then, we have
\begin{lemma}\label{lm3.4}

Let $T>0$ and $\gamma>1$. Let $(\rho,u,a)$ be the solution to
FBVP~\eqref{3.2}--\eqref{3.3b} for $(x,\tau)\in[0,1]\times[0,T]$.
Assume further that it holds for some $x_0\in(0,1)$
\bgr
 \mbox{$\frac12$}\rho_-(1-x)^\beta
    \le \rho(x,\tau)\le 2\rho_+(1-x)^\beta,
    \quad (x,\tau)\in[x_0,1]\times[0,T],   \label{rho.pm}
\\
  |\rho r^{N-1}u_x(x,\tau)|\le 2M_{0},\quad
 (x,\tau)\in [x_0,1]\times[0,T],   \label{m.ka}
\egr
where $\beta\in(0,1)$, $\rho_+=2\rho^*$, $\rho_-=\frac12\rho_*$ and
the constant $M_0>0$ is given by \ef{M_0}. Then, there is a time
$T_1\in(0,T]\cap(0,1]$ so
that $(\rho,u,a)$ satisfies
\bgr
 c_0x^{\frac{\gamma}{N(\gamma-1)}}
  \leq r(x,\tau)\le a(\tau)\le  2a_0,
  \quad  (x,\tau)\in [0,1]\times[0,T_1], \label{4.9ba}
\\
 \mbox{$\frac12$}\rho_0(x)\le \rho(x,\tau)\le 2\rho_0(x),
    \quad      (x,\tau)\in[x_0,1]\times[0,T_1], \label{rho.+.a}
\\
     \int_0^{x}((\rho^{\theta})_yr^{N-1})^2dy
  +\int_0^\tau\int_0^{x}
    ((\rho^{\frac{\gamma+\theta}2})_y r^{N-1})^2dyds
  \le CE_x,
    (x,\tau)\in (x_0,1)\times[0,T_1], \label{3.10a}
 \egr
with $E_x=:\int_0^{x}(u^2+((\rho^{\theta})_{y}r^{N-1})^2)(y,0)dy
+\int_0^{1}\rho_0^{\gamma-1}(y) dy$ and $C_{x_0,x}>0$ a constant.
\end{lemma}
\begin{proof}
First of all, it follows directly from \ef{m.ka} and \ef{rho.pm}
that
\bma
 \|u\|_{L^\infty([x_0,1]\times[0,T])}
 \le &
    (1-x_0)^{-1}\int_{x_0}^1|u|dx
  +\int_{x_0}^1|u_x|dx             \label{v_0a}
  \\
  \le &
    C_{x_0}(E_0^{1/2}+ M_0\rho_-^{-1})=:M_1, \label{6.v_0b}
\ema
which yields \ef{4.9ba} with the help of \ef{3.23a} and
\be
r(x,\tau)\le a(\tau) =a_0 +\int_0^{\tau} u(1,s)ds \le a_0 + TM_1\le 2
a_0, \label{r0a}
\ee
for $\tau\in[0, T_{1,a}]$ with
\be
T_{1,a}=:a_0M_1^{-1}.\label{T_{1,a}}
\ee
It follows from $\ef{3.2}_1$ that
\be
   \rho (x,\tau)
  =\rho_0(x)\exp\left\{-\int_0^\tau(\rho
   r^{N-1}u_x+(N-1)ur^{-1})(x,s)ds\right\},   \label{rhoport}
\ee
which together with \ef{m.ka}, \ef{4.9ba} and \ef{6.v_0b} yields
\ef{rho.+.a} for $\tau\in[0,T_{1,b}]$ with $T_{1,b}$ determined by
\be
 T_{1,b}=:\min\{\,T_{1,a},\
          (2M_1E_0^{\frac{1}{N(\gamma-1)}} x_0^{-\frac{\gamma}{N(\gamma-1)}}
           + 2M_0)^{-1}\ln2\,  \}.\label{T_{1,b}}
\ee

Differentiating $\ef{3.2}_1$ with respect to $x\in[x_0,1)$, substituting the resulted equation into $\ef{3.2}_2$ and using
the fact $ \frac{\partial r}{\partial \tau}=u,$  we have
 \be
  (u + r^{N-1}(\rho^\theta)_x)_\tau+(\rho^\gamma)_xr^{N-1} =0.  \label{3.15}
 \ee
Multiplying \ef{3.15} by $\phi(x)(u+ r^{N-1}(\rho^\theta)_x)$, where
$\phi\in C^\infty([0,1))$, $0\le \phi\le 1$, $\phi(y)=1$ for
$y\in[0,x]$ and $\phi(y)=0$ for $y>(1+\eta)x$ with $\eta>0$ small
enough, and integrating the resulted equation over
$[0,1]\times[0,\tau]$ by parts, we obtain
 \bma
 &\int_0^1\phi(x)(u+r^{N-1}(\rho^\theta)_x)^2 dx
  +\int_0^\tau\int_0^1
   \phi(x)((\rho^{\frac{\gamma+\theta}2})_x r^{N-1})^2dxds
\nnm\\
\le
 & C\int_0^1\phi(x)(u+r^{N-1}(\rho^\theta)_x)^2(x,0) dx
  +  \int_0^\tau\int_0^1 \phi_xu r^{N-1}\rho^\gamma dx ds+
  \mbox{$\frac{1}{\gamma-1}$}\int_0^1\phi(x)\rho^{\gamma-1}(x,0)dx
\nnm\\
 \le
 & C_{x_0,x}E_x + (2a_0)^{N-1}M_1\tau(\rho^*)^{2\gamma}, \label{3.16a}
 \ema
where we have used \ef{4.9ba} and \ef{rho.+.a}. Choose
$$
 T_{1,c}=:\min\{\ T_{1,a},\,T_{1,b},\ ((2a_0)^{N-1}M_1(\rho^*)^{2\gamma})^{-1}E_x\,\},
 \quad
 T_1=:\min\{\,T_{1,a},\ T_{1,b},\ T_{1,c}\,\},
$$
then the combination of \ef{3.16a} and \ef{3.4} yield \ef{3.10a}
for $\tau\in[0,T_{1}]$.
\end{proof}

\section{Boundary regularities}
This section is devoted to the boundary regularities of solutions to FBVP~\eqref{3.2}--\eqref{3.3b} . To this end, we first establish the regularities of solution $(\rho,u,a)$ away from the
symmetry center and the free boundary.
\begin{lemma} \label{boundary.a}
Under the assumptions of Lemma~\ref{lm3.4}, there is a time
$T_2\in(0,T_1]$ so that the solution $(\rho,u,a)$ satisfies for
$x_1\in(x_0,1)$ and $\tau\in[0,T_2]$ that
\bgr
     \|(\rho_x,u_x)(\tau)\|^2_{L^2([x_0,x_1])}
  + \int_0^{\tau}\|(u_{xx},u_{s})(s)\|^2_{L^2([x_0,x_1])}ds
 \le
  C_{4}\delta_4^2,  \label{boundary.a.1}
\\
   \|(\rho_{xx},u_{xx},u_{\tau})(\tau)\|^2_{L^2([x_0,x_1])}
  + \int_0^{\tau} \| (u_{xxx},u_{xs})(s)\|^2_{L^2([x_0,x_1])}ds
 \le
 C_{5}\delta_5^2, \label{boundary.a.2}
\\
   \|(\rho_{xxx},u_{xxx},u_{x\tau})(\tau)\|^2_{L^2([x_0,x_1])}
  + \int_0^{\tau} \| (u_{xxxx},u_{xxs})(s)\|^2_{L^2([x_0,x_1])}ds
 \le
 C_{6}\delta_6^2,  \label{boundary.a.3}
\egr
provided that $(\rho_0,u_0)\in H^3(I)$ with $[x_0,x_1]\subset
(x_0^-,x_1^+)$, where $x_0^-\in(0,x_0)$ and $I=:[x_0^-,x_1^+]$, where $C_{i}>0, (i=4, 5, 6)$ are constants
dependent of $x_0$ and $x_1$, but independent of $M_0$,
$\delta_4=\|(\rho_0,u_0)\|_{H^{1}(I)}$,
$\delta_5=\|(\rho_0,u_0)\|_{H^{2}(I)}$, and
$\delta_6=\|(\rho_0,u_0)\|_{H^{3}(I)}$.
In addition, it holds
\be
 |\rho r^{N-1}u_x(x,\tau)|\le M_{0,a},\quad
 |u(x,\tau)|\le M_{1,a},\quad
 (x,\tau)\in [x_0,\,x_1]\times[0,T_2],   \label{M0,a}
\ee
with $
 M_{0,a}=(2a_0)^{N-1}\rho^*(C_{4}\delta_4^2+C_{5}\delta_5^2)^{1/2}, \ \
 M_{1,a}=C_{x_0}(E_0^{1/2}+ M_{0,a}\rho_*^{-1}).
$
\end{lemma}
\begin{proof}
It is easy to verify that \ef{M0,a} follows from \ef{boundary.a.1},
\ef{boundary.a.2}, \ef{rho.+.a} and \ef{v_0a}. What left is to show
\ef{boundary.a.1}-\ef{boundary.a.3}. Rewrite $\ef{3.2}_2$ as
\be
  r^{1-N}u_{\tau}+(\rho^\gamma
   -\theta\rho^{1+\theta}r^{N-1}u_x)_x
    +(1-\theta)(N-1)(\rho^\theta)_x\frac{u}{r}
      -\theta(N-1)\rho^{\theta}(\frac{u}{r})_x=0,  \label{3.2re}
\ee
Take inner product between \ef{3.2re} and $\phi\rho^{1-N}u_{\tau}$, where $\phi=\psi^2(x)$ and $\psi\in C^\infty([0,1])$ satisfies $0\leq\psi\leq1, \psi=1$ for $x\in[(1-2\eta)x_0,(1+2\eta)x_1]$, and $\psi=0$ for $x\in[0,(1-3\eta)x_0]\cup[(1+3\eta)x_1,1]$ with a fixed constant $\eta\in(0,1)$ small enough so that $[(1-3\eta)x_0,(1+3\eta)x_1]\subset(x_0^-,x_1^+)$. By lemma~\ref {lm3.4} and a direct computations, it follows
\bma
  &\frac{d}{d\tau}\int_0^1 (
      \frac{\theta}{2} \phi\rho^{2+\theta-N}r^{N-1}u_x^2
         -\phi\rho^{\gamma+1-N}u_x)dx
    +\frac12\int_0^1 \phi\rho^{1-N}r^{1-N}u_{\tau}^2dx
\nnm\\
   \le
   &\,C_{x_0}(M_0+M_1+(\rho^*)^{\gamma-\theta})
      \int_0^1\phi\rho^{2+\theta-N}r^{N-1}u_x^2dx
         +C_{x_0,x_1}M_0M_1(\rho_*+\rho^*)^{1+\gamma-N}
\nnm\\
   &\,+C_{x_0,x_1}a_0^{N-1}(\rho_*+\rho^*)^{2\gamma+1-N}
       +C_{x_0,x_1}M_0^2(\rho_*+\rho^*)^{1+2\theta-N}
         +C_{x_0,x_1}M_1^2(\rho_*+\rho^*)^{2\theta-1-N}
\nnm\\
   &\,+C_{x_0,x_1^+}\rho_*^{1-N}E_{x_1^+}(M_0^2+M_1^2
        +(\rho_*+\rho^*)^{2\gamma+1-N-2\theta}),\label{boundary.a.1.1}
\ema
where $C_{x_0,x_1}>0$ is a generic constant dependent of $x_0$,
$x_1$, but independent of $M_0,\, M_1$.

Integrating \ef{boundary.a.1.1} over $[0,\tau]$ to get
\bma
&\int_0^1
    \phi \rho^{2+\theta-N} r^{N-1}u_x^2 dx
       +\int_0^{\tau} \int_0^1 \phi\rho^{1-N} r^{1-N}u^2_s dxds
\nnm\\
\le
&    C_{x_0,x_1}(M_0+M_1+(\rho^*)^{\gamma-\theta})\int_0^\tau \int_0^1\phi \rho^{2+\theta-N} r^{N-1}u_x^2 dx ds
   +\tau(C_{x_0,x_1}(M_0^2+M_1^2+a_0^{N-1})
   \nnm\\
& +C_{x_0,x_1}\rho_*^{1-N}E_{x_1^+}(M_0^2+M_1^2))+C_{x_0,x_1}\delta_4^2,  \label{boundary.a.1.2}
\ema
which together with Gronwall's inequality leads to
\be
    \int_0^1\phi \rho^{2+\theta-N} r^{N-1}u_x^2 dx+
       \int_0^{\tau} \int_0^1 \phi\rho^{1-N} r^{1-N}u^2_s dxds
          \le C_{x_0,x_1,1} \delta_4^2.  \label{boundary.a.1.3}
\ee
Due to the fact that $\phi(x)=1$ for
$x\in[(1-2\eta)x_0,(1+2\eta)x_1]$, it follows
\be
 \int_{(1-2\eta)x_0}^{(1+2\eta)x_1} u_x^2 dx
   +\int_0^{\tau}\int_{(1-2\eta)x_0}^{(1+2\eta)x_1}u^2_s dxds
 \le
     C_{x_0,x_1,2}\delta_4^2,  \label{boundary.a.1.4}
 \ee
for some constant $C_{x_0,x_1,2}>0$ and $\tau\in[0,T_{2,a}]$ with
$T_{2,a}$ chosen as
\bgr
 T_{2,a}=:\min\{T_1,\, K_1^{-1}\delta_4^2, \
      K_2^{-1} \ln2 \,\},
\nnm\\
K_1=:C_{x_0,x_1}(M_0^2+M_1^2+a_0^{N-1})
     +C_{x_0,x_1}\rho_*^{1-N}E_{x_1^+}(M_0^2+M_1^2),\nnm\\
 K_2=:C_{x_0,x_1}(M_0+M_1+(\rho^*)^{\gamma-\theta}).\nnm
\egr
In addition, it follows from \ef{boundary.a.1.4}, $\ef{3.2}_2$ and
Lemma~\ref{lm3.4} that
\be
\int_0^{\tau}\int_{(1-2\eta)x_0}^{(1+2\eta)x_1}
 (u_{xx}^2+\rho_{xs}^2)dxds
 \le
    C_{x_0,x_1}\delta_4^2.   \label{boundary.a.1.5}
\ee
The combination of \ef{boundary.a.1.4}--\ef{boundary.a.1.5} and \ef{3.10a} gives rise to
\be
  \int_{(1-2\eta)x_0}^{(1+2\eta)x_1}
   ( u^2_x + \rho_x^2)dx
 + \int_0^{\tau}\int_{(1-2\eta)x_0}^{(1+2\eta)x_1}
   ( u_{xx}^2 +\rho_{xs}^2+u^2_{s})dx ds
 \le
 C_{x_0,x_1}\delta_4^2,     \label{boundary.a.1.6}
\ee
for $\tau\in[0,T_{2,a}]$, which implies \ef{boundary.a.1} for any $T_2\le
T_{2,a}$.
\par

The higher order regularities of the solution can be obtained by
applying the similar arguments as the proof of \ef{boundary.a.1.6}. Indeed,
differentiating $\ef{3.2re}$ with respect to $\tau$ gives
\bma
 &r^{1-N}u_{\tau\tau} -(1-N) r^{-N}u u_\tau
  + (\rho^\gamma -\theta\rho^{\theta+1}r^N u_x)_{x\tau}\nnm\\
  &+ (1-\theta)(N-1)((\rho^\theta)_x\frac{u}{r})_\tau
  - \theta(N-1)\mbox{$(\rho^\theta(\frac{u}{r})_x)_\tau$} = 0,
  \label{boundary.a.2.1}
\ema
taking inner product between \ef{boundary.a.2.1} and $\phi u_\tau$ over
$[0,1]$, where $\phi=\psi^2(x)$ and $\psi\in C^\infty([0,1])$
satisfies $0\le \psi(x)\le 1$, $\psi(x)=1$ for $x\in
[(1-\eta)x_0,(1+\eta)x_1]$, and $\psi(x) =0$ for $x\in
[0,(1-2\eta)x_0]\cup[(1+2\eta)x_1,1]$, we can obtain
\bma
 &  \frac12\frac{d}{d\tau}\int_0^1 \phi r^{1-N}u^2_\tau dx
   +\frac{\theta}{2}\int_0^1 \phi \rho^{\theta+1}r^{N-1}u^2_{x\tau}dx
\nnm\\
  =& \int_0^1 \phi_x (\rho^\gamma -\theta\rho^{\theta+1}r^{N-1}u_x)_{\tau}u_{\tau}dx
    +\frac{N-1}{2}\int_0^1 \phi r^{-N}u u^2_\tau dx
   +\theta(N-1)\int_0^1 \phi \mbox{$(\rho^{\theta}(\frac{u}{r})_x)_\tau$} u_{\tau}dx
\nnm\\
   & +\int_0^1 \phi ((\rho^\gamma)_{\tau} -\theta(\rho^{\theta+1}r^{N-1})_{\tau}u_x) u_{x\tau}dx
     -\frac{\theta}{2}\int_0^1 \phi \rho^{\theta+1}r^{N-1}u^2_{x\tau}dx
\nnm\\
   &+(1-\theta)(1-N)\int_0^1 \phi \mbox{$((\rho^{\theta})_x\frac{u}{r})_\tau$} u_\tau dx
\nnm\\
 \le
 &\,C_{x_0,x_1}(1+M_1+(\rho_*+\rho^*)^{\theta-1})
     \int_0^1 \phi r^{1-N}u^2_\tau dx
       +C_{x_0,x_1}(1+(\rho^*)^{\theta+1}
 \nnm\\
    &+(\rho_*+\rho^*)^{\theta-1})
       \int_{(1-2\eta)x_0}^{(1+2\eta)x_1} u^2_{\tau} dx
        +C_{x_0,x_1}(\rho_*+\rho^*)^{\theta-1}
         \int_{(1-2\eta)x_0}^{(1+2\eta)x_1} \rho_{x\tau}^2 dx
\nnm\\
    &+C_{x_0,x_1,3}(1+(\rho^*)^{2(\gamma+1)}+(\rho^*)^4)(1 +M_0^4+M_1^4),   \label{boundary.a.2.2}
\ema
which together with Lemma~\ref{lm3.4} and \ef{boundary.a.1.6} yields
 \be
   \int_0^1\phi r^{1-N}u^2_{\tau} dx
 + \int_0^\tau\int_0^1 \phi \rho^{\theta+1}r^{N-1}u^2_{xs}dx ds
 \le  C_{x_0,x_1}\delta_5^2, \quad \tau\in[0,T_{2,d}],   \label{boundary.a.2.3}
 \ee
with $T_{2,b}$ chosen as
 \bgr
 T_{2,b}=:\min\{\,T_{2,a},\ K_3^{-1}\delta_4^2,\ (C_{x_0}(1+M_1+(\rho_*+\rho^*)^{\theta-1}))^{-1}\ln2\,\},
 \nnm\\
K_3=:C_{x_0,x_1,3}(1+(\rho^*)^{2(\gamma+1)}+(\rho^*)^4)
     (1+M_0^4+M_1^4).
  \nnm
  \egr
This and $\ef{3.2}_2$, \ef{4.9ba}--\ef{rho.+.a}
imply that
\be
 \int_{(1-\eta)x_0}^{(1+\eta)x_1}(u_{xx}^2+u^2_{\tau})dx
 + \int_0^{\tau}\int_{(1-\eta)x_0}^{(1+\eta)x_1} u^2_{xs}dx ds
 \le  C_{x_0,x_1}\delta_5^2, \quad \tau\in[0,T_{2,b}].    \label{boundary.a.2.4}
 \ee

Meanwhile, taking inner product between $\ef{3.15}_{x}$ and
$\phi(x)(u+r^{N-1}(\rho^{\theta})_x)_{x}$ over $[0,1]\times[0,\tau]$, and using
Lemma~\ref{lm3.4}, we can obtain
\bma
 &\int_0^1\phi (u+r^{N-1}(\rho^{\theta})_x)_x^2 dx
  +\frac{\gamma}{\theta}\int_0^\tau\int_0^1
   \phi \rho^{\gamma-\theta}(u+r^{N-1}(\rho^{\theta})_x)_x^2dxds
\nnm\\
=
 & \int_0^1\phi (u+r^{N-1}(\rho^{\theta})_x)_x^2(x,0) dx
  + \frac{2\gamma}{\theta}\int_0^\tau\int_0^1
     \phi \rho^{\gamma-\theta}(u+r^{N-1}(\rho^{\theta})_x)_x u_xdxds
\nnm\\
 &-\frac{2\gamma}{\theta}\int_0^\tau\int_0^1\phi r^{N-1}
     (u+r^{N-1}(\rho^{\theta})_x)_x
       (\rho^{\gamma-\theta})_x (\rho^{\theta})_xdxds
\nnm\\
\le
& C_{x_0,x_1}((\rho^*)^{\gamma-\theta}+(\rho^*)^{\gamma-1}
   +(\rho_*+\rho^*)^{\gamma-3}+(\rho_*+\rho^*)^{2(\theta-1)}
\nnm\\
 &+(\rho_*)^{\theta-\gamma}\int_0^\tau
  \max[(\rho^{\gamma})_x r^{N-1}]^2ds)\delta_4^2.
  \label{boundary.a.2.5}
\ema
Integrating \ef{3.15} over $[0,\tau]$ to get
\be
  r^{N-1}(\rho^{\theta})_x(x,\tau)
   =r^{N-1}(x,0)(\rho_0^{\theta})_x
    -\int_0^\tau (\rho^{\gamma})_x r^{N-1}(x,s)ds
      -u(x,\tau)+u(x,0),  \label{boundary.a.2.6}
\ee
one deduces from \eqref{rho.+.a}, \ef{6.v_0b} and \ef{boundary.a.2.5} that for
any $(x,\tau)\in[(1-2\eta)x_0,(1+2\eta)x_1]\times[0,T_{2,b}]$,
\bma
 &\int_0^\tau[(\rho^\gamma)_xr^{N-1}]^{2}(x,s)ds\nnm\\
\leq&
    C\tau(\rho^*)^{2(\gamma-\theta)}(1+M_1^2)
     +C(\rho^*)^{2(\gamma-\theta)}
     \int_0^\tau\int_0^s[(\rho^\gamma)_xr^{N-1}]^{2}(x,s)dzds,
       \label{boundary.a.2.7}
 \ema
 which implies for $(x,\tau)\in[(1-2\eta)x_0,(1+2\eta)x_1]\times[0,T_{2,c}]$ that
 \be
 \int_0^\tau[(\rho^\gamma)_xr^{N-1}]^{2}(x,s)ds
 \leq
  2\ln2,  \quad  \tau\in[0,T_{2,c}],\label{boundary.a.2.8}
 \ee
with $T_{2,c}$ chosen as
\[
 T_{2,c} = :\min\{\,T_{2,b},\
      (C(\rho^*)^{2(\gamma-\theta)}(1+M_1^2))^{-1}\ln2,\
      (C(\rho^*)^{2(\gamma-\theta)})^{-1}\ln2\,\}.
\]
Using Lemma~\ref{lm3.1}, \ef{3.10a} and \ef{boundary.a.1} that
\be
  \int_{(1-\eta)x_0}^{(1+\eta)x_1}\rho_{xx}^2dx
 + \int_0^{\tau}\int_{(1-\eta)x_0}^{(1+\eta)x_1}\rho_{xx}^2dxds
  \le
   C_{x_0,x_1}\delta_5^2,\quad  \tau\in[0,T_{2,c}].\label{boundary.a.2.10}
\ee
The combination of \ef{boundary.a.2.4}, \ef{boundary.a.2.10}, \ef{boundary.a.1}
and $\ef{3.2}_2$ gives rise to
\be
  \int_{(1-\eta)x_0}^{(1+\eta)x_1}
   (u^2_{\tau}+u_{xx}^2+\rho_{xx}^2)dx
 + \int_0^{\tau}\int_{(1-\eta)x_0}^{(1+\eta)x_1}
   (u_{xxx}^2+\rho_{xx}^2+u^2_{xs})dxds
 \le
  C_{x_0,x_1}\delta_5^2,  \label{boundary.a.2.11}
\ee
for $\tau\in[0,T_{2,c}]$, which implies \ef{boundary.a.2} for any $T_2\le
T_{2,c}$.
\par

Differentiating $r^{N-1}\ef{3.2re}$ with respect to $x$ to get
\bma
 &u_{xxx\tau} + (r^{N-1}(\rho^\gamma
 -\theta\rho^{\theta+1}r^{N-1}u_x)_{x})_{xxx}\nnm\\
  &+(N-1)(1-\theta)(r^{N-2}(\rho^{\theta})_x u)_{xxx}
  -\theta(N-1)(r^{N-1}\mbox{$\rho^{\theta}(\frac{u}{r})_x$})_{xxx} = 0,   \label{boundary.a.3.1}
 \ema
and taking inner product between \ef{boundary.a.3.1} and $\phi(x)u_{xxx}$
over $[0,1]\times[0,\tau]$, where $\phi=\psi^2(x)$ and $\psi\in
C_0^\infty([0,1])$ satisfies $0\le \psi(x)\le 1$, $\psi(x)=1$ for
$x\in [x_0,x_1]$, and $\psi(x) =0$ for $x\in
[0,(1-\eta)x_0]\cup[(1+\eta)x_1,1]$.
We can obtain
\bma
 & \frac12\frac{d}{d\tau}\int_0^1\phi u_{xxx}^2dx
  -\int_0^1\phi (r^{N-1}(\rho^\gamma
  -\rho^{\theta+1}r^{N-1}u_x)_{x})_{xx}u_{xxxx}dx
\nnm\\
 =&\,-\int_0^1\phi_x (r^{N-1}(\rho^\gamma
     -\rho^{\theta+1}r^{N-1}u_x)_{x})_{x}u_{xxxx}dx
     -\int_0^1\phi_{xx} (r^{N-1}(\rho^\gamma
     -\rho^{\theta+1}r^{N-1}u_x)_{x})_{x}u_{xxx}dx
\nnm\\
 & +\theta(N-1)\int_0^1\phi
   (r^{N-1}\mbox{$\rho^{\theta}(\frac{u}{r})_x$})_{xxx}u_{xxx}dx
    +(N-1)(1-\theta)\int_0^1\phi_x
    u_{xxx}(r^{N-2}(\rho^{\theta})_x u)_{xx}dx
\nnm\\
 &+(N-1)(1-\theta)\int_0^1\phi u_{xxxx}(r^{N-2}(\rho^{\theta})_x u)_{xx} dx.
    \label{boundary.a.3.2}
  \ema
Lemma~\ref{lm3.4}, \ef{3.2}, \ef{boundary.a.1.6} and
\ef{boundary.a.2.11} lead to that
\bma
 & \frac{d}{d\tau}\int_0^1\phi (u_{xxx}^2+ \gamma\rho^{\gamma-3}\rho_{xxx}^2)dx
   +\int_0^1\phi \rho^{\theta+1}r^{2(N-1)}u_{xxxx}^2dx
\nnm\\
 \le &\,
   C_{x_0,x_1}(1+\delta_4^2+\delta_5^2
        +\|u_{xx}(\tau)\|_{L^\infty([(1-\eta)x_0,(1+\eta)x_1])})
    \int_0^1\phi(\rho^{\gamma-3}\rho_{xxx}^2 + u_{xxx}^2)dx
\nnm\\
  & +C_{x_0,x_1}(1+\delta_4^2+\delta_5^2)^2
    +C_{x_0,x_1}\|u_{xx}(\tau)\|_{L^\infty([(1-\eta)x_0,(1+\eta)x_1])}\delta_5^2,
\nnm
\ema
 we apply the Gronwall's inequality, and the fact $\psi(x)=1$ for
$x\in [x_0,x_1]$, we can obtain
\be
   \int_{x_0}^{x_1}
     (u_{xxx}^2 + \rho_{xxx}^2)dx
  + \int_0^\tau\int_{x_0}^{x_1}
     u_{xxxx}^2dxds
 \le
   C_{x_0,x_1}\delta_6^2,\quad   \tau\in[0,T_{2,d}],    \label{boundary.a.3.3}
 \ee
with $T_{2,d}$ determined by
$$
T_{2,d}=:\min\{\,T_{2,c},\
     (C_{x_0,x_1}(1+\delta_4^2+\delta_5^2)^2)^{-1}\delta_6^2,\
     (C_{x_0,x_1}(1+\delta_4^2+\delta_5^2))^{-1}\ln2\}.\nnm
$$
By \ef{boundary.a.3.3} and $\ef{3.2}_2$, we have
\be
 \int_{x_0}^{x_1}u_{x\tau}^2dx
  +  \int_0^\tau \int_{x_0}^{x_1} u_{xxs}^2 dxds
  \le   C_{x_0,x_1}\delta_6^2,  \quad \tau\in[0,T_{2,d}].       \label{boundary.a.3.4}
\ee
The combination of \ef{boundary.a.3.3} and \ef{boundary.a.3.4} yields \ef{boundary.a.3} with
$T_2=T_{2,d}$.
\end{proof}

\begin{lemma} \label{boundary.b}
Under the same assumptions as Lemma~\ref{lm3.4}, there is a time
$T_3\in(0, T_2]$ so that it holds for
$(x_2,\tau)\in(x_0,x_1)\times[0,T_3]$
\bgr
  \int_{x_2}^1 u^{2k}(x,\tau)dx
  +\int_0^\tau\int_{x_2}^1 \rho^{\theta+1}u^{2k-2}r^{2N-2}u_x^2dxds
  \le C_7,\quad k=1,2,\cdots , 2m, \label{boundary.b.1}
 \\
   \int_{x_2}^1(\rho^{\theta+1}u_x^2+\rho^{\theta-1}u^2)dx
   +\int_0^\tau\int_{x_2}^1u^2_{s}dxds
   \le C_8+\int_0^\tau\int_{x_2}^1 \rho^{\theta+3}u_x^4 dxds, \label{boundary.b.2}
 \\
   \int_{x_2}^1 \rho^{\theta+3}u_x^4 dx
   \le C_9+C(\int_{x_2}^1(\rho^{\theta+1}u_x^2+\rho^{\theta-1}u^2+u^2_{\tau})dx)^2,
   \label{boundary.b.3}
 \\
   \int_{x_2}^1 u^2_{\tau}dx
   +\int_0^\tau\int_{x_2}^1 (\rho^{\theta+1} r^{2N-2}u^2_{xs}
   +\rho^{\theta-1}r^{-2}u^2_{s})dxds
   \le C_{10}, \label{boundary.b.4}
  \egr
 where $C_i=C_i(\|\rho_0\|_{H^1([x_0, 1])}, \|u_0\|_{H^2([x_0,1])})>0$ are
 constants,
 $i=7, 8, 9, 10$.
\end{lemma}
\begin{proof}
We apply the arguments used in [3] to show \ef{boundary.b.1}. First
we consider the case of $k=1$. From Lemma~\ref{lm3.1}, we obtain
\ef{boundary.b.1} with $k=1$ easily. Assume \ef{boundary.b.1} holds
for $k=l-1$,
\be
  \int_{x_2}^1u^{2(l-1)}dx
  +\int_0^{\tau}\int_{x_2}^1\rho^{\theta+1}u^{2(l-1)-2}r^{2N-2}u_x^2dxds
  \le C.
\ee
Now we need to prove \ef{boundary.b.1} holds for $k=l$. Multiplying $\ef{3.2}_2$ by $\phi u^{2l-1}$ and integrating over $x$ from 0 to 1, where $\phi=\psi^2(x)$ and $\psi\in C^\infty([0,1])$ satisfies $0\leq\psi\leq1$, $\psi=1$ for $x\in[x_2,1]$, and $\psi=0$ for $x\in[0,x_0]$. We can obtain
\bma
 & \frac{d}{d\tau}\int_0^1\frac{1}{2l}\phi u^{2l}dx
 \nnm\\
   =&-\theta\int_0^1\phi\rho^{\theta+1}(r^{N-1}u^{2l-1})_x(r^{N-1}u)_xdx
     +\int_0^1\phi\rho^{\gamma}(r^{N-1}u^{2l-1})_xdx
 \nnm\\
   &+(N-1)\int_0^1\phi\rho^{\theta}(r^{N-2}u^{2l})_xdx
   -\theta\int_0^1\rho^{\theta+1}r^{N-1}u^{2l-1}\phi_x(r^{N-1}u)_xdx
 \nnm\\
   &+\int_0^1\rho^{\gamma}r^{N-1}u^{2l-1}\phi_xdx
   +(N-1)\int_0^1\rho^{\theta}r^{N-2}u^{2l}\phi_xdx
   \nnm\\
 :=&H_1+H_2+H_3+H_4+H_5+H_6.
     \label{boundary.b.1.1}
\ema
Set
\be
   B_1^2=\rho^{\theta+1}u^{2l-2}r^{2N-2}u_x^2\geq0,\quad B_2^2=\rho^{\theta-1}r^{-2}u^{2l}\geq0,\nnm
\ee
thus
\bma
   H_1+H_3=&
   -\theta(2l-1)\int_0^1\phi B_1^2dx-
     2(N-1)(\theta-1)l\int_0^1\phi B_1B_2dx
     \nnm\\
      &+((N-1)(N-2)-\theta(N-1)^2)\int_0^1\phi B_2^2dx.\nnm
\ema
Inserting this to \ef{boundary.b.1.1} and using Young's inequality, we get
 \bma
 &\frac{d}{d\tau}\int_0^1\frac{1}{2l}\phi u^{2l}dx
    +\theta(2l-1)\int_0^1\phi B_1^2dx
    \nnm\\
 \le
 &\varepsilon\int_0^1\phi B_1^2dx+C_{\varepsilon}\int_0^1\phi B_2^2dx+
   \int_0^1\phi\rho^{\gamma}(r^{N-1}u^{2l-1})_xdx+|H_4|+|H_5|+|H_6|.
     \label{boundary.b.1.2}
 \ema
By a direct computation, it follows
\bma
 &\frac{d}{d\tau}\int_0^1\phi u^{2l}dx+
  \int_0^1\phi\rho^{\theta+1}u^{2l-2}r^{2N-2}u_x^2dx
 \nnm\\
\le
 &C_{x_0}(1+M_{1,a}^{2l}+M_{1,a}^{2l-1}M_{0,a}+M_1^{2l-2})
   +C_{x_0}(1+M_1^2)\int_0^1\phi u^{2l}dx,\nnm
\ema
using Gronwall's inequality and the fact that
$\phi=1$ for $x\in[x_2,1]$, we get
\be
  \int_{x_2}^1u^{2l}dx+
    \int_0^{\tau}\int_{x_2}^{1}\rho^{\theta+1}u^{2l-2}r^{2N-2}u_x^2dxds
      \le C, \quad\tau\in[0,T_{3,a}],  \label{boundary.b.1.3}
\ee
with $T_{3,a}$ determined by
$$
T_{3,a}=:\min\{(C_{x_0}(1+M_{1,a}^{2l}+M_{1,a}^{2l-1}
M_{0,a}+M_1^{2l-2}))^{-1},(C_{x_0}(1+M_1^2))^{-1}\ln2\},\nnm
$$
we get \ef{boundary.b.1} immediately.

To show \ef{boundary.b.2} we multiplying $\ef{3.2}_2$
by $\phi r^{N-1}u_{\tau}$, where $\phi=\psi^2(x)$ and
$\psi\in C^\infty([0,1])$ satisfies $0\leq\psi\leq1$,
$\psi=1$ for $x\in[x_2,1]$, and $\psi=0$ for $x\in[0,x_0]$,
integrate it over $[0,1]\times[0,\tau]$ to have
\bma
  &\int_0^{\tau}\int_0^1\phi u^2_{s}(x,s)dxds
\nnm\\
  =&\int_0^{\tau}\int_0^1\phi\rho^{\gamma}(u_{s}r^{N-1})_xdxds
     -\theta\int_0^{\tau}\int_0^1\phi\rho^{\theta+1}
      (r^{N-1}u)_x(u_{s}r^{N-1})_xdxds
\nnm\\
   &+(N-1)\int_0^{\tau}\int_0^1\phi\rho^{\theta}(r^{N-2}uu_{s})_xdxds
     +\int_0^{\tau}\int_0^1\phi_x\rho^{\gamma}u_{s}r^{N-1}dxds
\nnm\\
   &-\theta\int_0^{\tau}\int_0^1\phi_x\rho^{\theta+1}
     (r^{N-1}u)_xu_{s}r^{N-1}dxds
     +(N-1)\int_0^{\tau}\int_0^1\phi_x\rho^{\theta}r^{N-2}uu_{s}dxds
  \nnm\\
  \leq&
  C\delta_4^2+C\int_0^{\tau}\int_0^1\phi\rho^{\theta+3}u_x^4dxds
  -C\int_0^1\phi(\rho^{\theta+1}u_x^2+\rho^{\theta-1}u^2)dx,\nnm
\ema
which implies
\bma
  &\int_0^{\tau}\int_{x_2}^1u_{s}^2(x,s)dxds
    +\int_{x_2}^1(\rho^{\theta+1}u_x^2+\rho^{\theta-1}u^2)dx\nnm\\
 \le&
    C+C\int_0^{\tau}(\int_{x_0}^{x_2}+
    \int_{x_2}^1)\phi\rho^{\theta+3}u_x^4dxds
\nnm\\
  \le&
    C+C\int_0^{\tau}\int_{x_2}^1\rho^{\theta+3}u_x^4dxds,
\ema
last we get \ef{boundary.b.2}, where we use the Lemma~\ref{lm3.2}, Lemma~\ref{boundary.a}, \ef{3.10a} and $ m>\frac{3}{4(1+\beta\theta-\beta)}$.

Now we begin to prove \ef{boundary.b.3}.
Multiplying $\ef{3.2}_2$ with $\phi$, where $\phi=\psi^2(x)$ and $\psi\in C^\infty([0,1])$ satisfies $0\le\psi\le1$, $\psi=1$ for $x\in[x_2,1]$, and $\psi=0$ for $x\in[0,x_0]$, then integrating over $[x,1]$, we have
\bma
   \phi u_x
   =&\frac{1}{\theta}\phi r^{1-N}\rho^{\gamma-\theta-1}
    -(N-1)\phi\rho^{-1}ur^{-N}
    +\frac{1}{\theta}(N-1)\phi\rho^{-1}ur^{-N}
\nnm\\
   &\,+\frac{1}{\theta}(N-1)\rho^{-1-\theta}r^{1-N}
    \int_x^1\phi(\frac{u}{r})_y\rho^{\theta}dy
     -\frac{1}{\theta}\rho^{-1-\theta}r^{1-N}
        \int_x^1\phi r^{1-N}u_{\tau}dy
\nnm\\
   &\,+\frac{1}{\theta}\rho^{-1-\theta}r^{1-N}\int_x^1
    \phi_y(\rho^{\gamma}-\theta\rho^{\theta+1}(r^{N-1}u)_y)dy
\nnm\\
   &\,+\frac{1}{\theta}(N-1)\rho^{-1-\theta}r^{1-N}
     \int_x^1\rho^{\theta}\phi_y\frac{u}{r}dy,\nnm
\ema
integrating it over $[0,1]$, we have
\bma
   &\int_0^1\rho^{\theta+3}\phi^4u_x^4dx
\nnm\\
  \le
   &\, C\int_0^1\phi^4\rho^{4\gamma-3\theta-1}dx
      +C\int_0^1\phi^4\rho^{\theta-1}u^4dx
        +C\int_0^1\rho^{-3\theta-1}(\int_x^1\phi
         (\frac{u}{r})_y\rho^{\theta}dy)^4dx
\nnm\\
   &\,+C\int_0^1\rho^{-3\theta-1}(\int_x^1r^{1-N}\phi u_{s}dy)^4dx
     +C\int_0^1\rho^{-3\theta-1}(\int_x^1\rho^{\theta}\phi_y\frac{u}{r}dy)^4dx
\nnm\\
   &\,+C\int_0^1\rho^{-3\theta-1}(\int_x^1\phi_y
    (\rho^{\gamma}-\theta\rho^{\theta+1}(r^{N-1}u)_y)dy)^4dx
    \nnm\\
\le& C+C(\int_0^1\phi^2\rho^{\theta+1}u_x^2dx)^2
      +C(\int_0^1\phi^2\rho^{\theta-1}u^2dx)^2
      +C(\int_0^1\phi^2u_{\tau}dy)^2
\ema
which implies
\be
   \int_0^1\phi^4\rho^{\theta+3}u_x^4dx
    \le
     C+C(\int_0^1\phi^2(\rho^{\theta-1}u^2
      +u_{\tau}^2+\rho^{\theta+1}u_x^2)dx)^2,\nnm
\ee
and
\be
   \int_{x_2}^1\rho^{\theta+3}u_x^4d
    \leq
     C+C(\int_{x_2}^1(\rho^{\theta-1}u^2
      +u_{\tau}^2+\rho^{\theta+1}u_x^2)dx)^2.\nnm
\ee

Finally we prove \ef{boundary.b.4}. Differentiating equation ${\ef{3.2}_2}$ respect to ${\tau}$,
\be
     u_{\tau\tau}
    =(r^{N-1}(\theta\rho^{\theta+1}(r^{N-1}u)_x
      -\rho^{\gamma})_x)_{\tau}-(N-1)(r^{N-2}u(\rho^{\theta})_x)_{\tau}
    =I+J,     \label{boundary.b.4.1}
\ee
where
\be
   I=r^{N-1}(\theta\rho^{\theta+1}(r^{N-1}u_{\tau}))_x
     -(N-1)r^{N-2}u_{\tau}(\rho^{\theta})_x,
     \nnm\\
\ee
\bma
   J=&(r^{N-1}(\theta\rho^{\theta+1}(r^{N-1}u)_x
      -\rho^{\gamma})_x)_{\tau}-
      r^{N-1}(\theta\rho^{\theta+1}(r^{N-1}u_{\tau}))_x\nnm\\
      &-(N-1)(N-2)r^{N-3}u^2(\rho^{\theta})_x
      -(N-1)r^{N-2}u(\rho^{\theta})_{x\tau}.\nnm
\ema

Multiplying \ef{boundary.b.4.1} by $\phi u_{\tau}$ and integrating over $x$ from 0 to 1, where $\phi=\psi^2(x)$ and $\psi\in C^\infty([0,1])$ satisfies $0\leq\psi\leq1$, $\psi=1$ for $x\in[x_2,1]$, and $\psi=0$ for $x\in[0,x_0]$, we have
\bma
   &\frac{d}{d\tau}\int_0^1 \frac{\phi}{2}u_{\tau}^2dx
     +(\mbox{$\frac{1}{N}$}+\theta-1)\int_0^1
      \phi\rho^{\theta+1}[(r^{N-1}u_{\tau})_x]^2dx
       +\mbox{$\frac{N-1}{N}$}\int_0^1\phi\rho^{\theta+1}(r^{N-1}u_{x\tau}
        -\rho^{-1}r^{-1}u_{\tau})^2dx
\nnm\\
   =&-\int_0^1\theta\phi_x\rho^{\theta+1}r^{N-1}u_{\tau}(r^{N-1}u_{\tau})_xdx
    +(N-1)\int_0^1\rho^{\theta}\phi_xr^{N-2}u_{\tau}^2dx
      +\int_0^1\phi Ju_{\tau}dx
      \nnm\\
\le
  &\,\varepsilon\int_0^1\phi\rho^{\theta+1}r^{2(N-1)}u_{x\tau}^2dx
      +\varepsilon\int_0^1\phi\rho^{\theta-1}u_{\tau}^2r^{-2}dx
       +C_\varepsilon\int_0^1\phi\rho^{\theta+3}u_x^4dx
\nnm\\
  &\,+C\int_0^1\phi u^{4m}dx+C\int_0^1\phi\rho^{(\theta-1)\frac{m}{m-1}}
       +C\int_0^1\phi \rho^{3(1-\theta)}+C,\nnm
\ema
after integrating over [0,$\tau$] and choosing proper $\varepsilon$, which implies
\bma
 &\int_{x_2}^1u_{\tau}^2dx
   +\int_0^{\tau}\int_{x_2}^1(\rho^{\theta+1}r^{2N-2}u_{xs}^2
     +\rho^{\theta-1}r^{-2}u_{s}^2)dxds
\nnm\\
  \le
  &\,C\int_0^{\tau}\int_{x_2}^1\rho^{\theta+3}u_x^4dxds
        +\int_{x_0}^1u_{\tau}^2(x,0)dx+C
\nnm\\
   \le
   &\,C\int_0^{\tau}(\int_{x_2}^1(\rho^{\theta-1}u^2
     +u_{s}^2+\rho^{\theta+1}u_x^2)dx)^2ds+C,\nnm
\ema
Using \ef{boundary.b.2} and above we have
\bma
  &\int_{x_2}^1(\rho^{\theta+1}u_x^2
    +\rho^{\theta-1}u^2+u_{\tau}^2)dx
     +\int_0^{\tau}\int_{x_2}^1(\rho^{\theta+1}r^{2N-2}u_{xs}^2
      +\rho^{\theta-1}r^{-2}u_{s}^2)dxds
\nnm\\
  \le
  &\,C_{11}+C_{12}\int_0^{\tau}(\int_{x_2}^1(\rho^{\theta+1}u_x^2
    +\rho^{\theta-1}u^2+u_{s}^2)dx)^2ds,\nnm
\ema
by Gronwall's inequality, we get
\be
   \int_{x_2}^1(\rho^{\theta+1}u_x^2
      +\rho^{\theta-1}u^2+u_{\tau}^2)dx+
       \int_0^{\tau}\int_{x_2}^1(\rho^{\theta+1}r^{2N-2}u_{xs}^2
        +\rho^{\theta-1}r^{-2}u_{s}^2)dxds
  \leq2C_{11},\quad 0\leq\tau\leq T_3,\nnm
\ee
where
\be
T_3=:\min\{T_2,T_{3,a},\frac{1}{2C_{11}C_{12}}\},\nnm
\ee
thus we get \ef{boundary.b.4}.
\end{proof}

\begin{lemma} \label{boundary.c}
Under the same assumptions as Lemma~\ref{lm3.4}, the solution $(\rho,u,a)$ satisfies for the $x_2\in(x_0,x_1)$ and $\tau\in[0,T_3]$ that
\be
  \int_{x_2}^1|u_x|^{\lambda_0}dx\le C, \label{boundary.c.1}
\ee
where ${\lambda_0}$ is a constant satisfying:
\be
 1<\lambda_0<\min\{\frac{4m}{4m\beta+1},\frac{1}{\beta(\theta+1)}\}.
  \label{boundary.c.2}
\ee
In addition, it holds
\be
 \mid \rho r^{N-1}u_x(x,\tau)\mid\leq M_{0,b},
  \quad\mid u(x,\tau)\mid\leq M_{1,b},
   \quad(x,\tau)\in[x_2,1]\times[0,T_3],\label{boundary.c.3}
\ee
where $M_{0,b}$ and $M_{1,b}$ are given by \ef{boundary.c.3.4} and \ef{boundary.c.3.2} respectively, and
\be
  \rho\in C^0([x_2,1]\times[0,T_3]),\quad
  u\in C^0([x_2,1]\times[0,T_3]).\label{boundary.c.4}
\ee
\end{lemma}

\begin{proof}
From $\ef{3.2}_2$ and boundary conditions, we have
\bma
   u_x
  =&\frac{1}{\theta}r^{1-N}\rho^{\gamma-\theta-1}
    -(N-1)\rho^{-1}ur^{-N}
    +\frac{1}{\theta}(N-1)\rho^{-1}ur^{-N}
\nnm\\
   &+\frac{1}{\theta}(N-1)\rho^{-1-\theta}r^{1-N}
     \int_x^1(\frac{u}{r})_y\rho^{\theta}dy
     -\frac{1}{\theta}\rho^{-1-\theta}r^{1-N}\int_x^1r^{1-N}u_{\tau}dy, \label{boundary.c.1.1}
\ema
from \ef{boundary.c.1.1}, we obtain
\bma
   \int_{x_2}^1|u_x|^{\lambda_0}dx
   \le
   &\,C\int_{x_2}^1(1-x)^{\lambda_0\beta(\gamma-\theta-1)}dx
     +C(\int_{x_2}^1\rho^{-\frac{4m}
     {4m-\lambda_0}\lambda_0}dx)^{\frac{4m-\lambda_0}{4m}}
    (\int_{x_2}^1u^{4m}dx)^{\frac{\lambda_0}{4m}}
\nnm\\
   &\,
    +C\|u_{\tau}\|_{L^2[x_2,1]}^{\lambda_0}\int_{x_2}^1
      (1-x)^{\lambda_0(\frac{1}{2}-\beta(\theta+1))}dx
\nnm\\
    &\,+C(\int_{x_2}^1\rho^{\theta+3}u_y^4dy)^{\frac{\lambda_0}{4}}
      \int_{x_2}^1|\rho^{-\theta-1}
     (\int_x^1\rho^{\theta-1}dy)^{\frac{3}{4}}|^{\lambda_0}dx
\nnm\\
    &\,+C(\int_{x_2}^1\rho^{\theta-1}u^2dy)^{\frac{\lambda_0}{2}}
     \int_{x_2}^1|\rho^{-\theta-1}(\int_x^1\rho^{\theta-1}dy)
      ^{\frac{1}{2}}|^{\lambda_0}dx\leq C_{13},
\label{boundary.c.1.2}
\ema
so we obtain \ef{boundary.c.1}, where we use the fact $\lambda_0<\min\{\frac{1}{\beta(\theta+1)}, \frac{4m}{4m\beta+1}\}$.
Then we have
\bma
  \|u\|_{L^\infty([x_2,1]\times[0,T_3])}
  \le
  &\,\frac{1}{1-x_2}\int_{x_2}^1|u|dx
     +\int_{x_2}^1|u_x|dx
\nnm\\
   \le
   &\,C_{x_2}E_0^{\frac{1}{2}}+
   (\int_{x_2}^1|u_x|^{\lambda_0}dx)^
{\frac{1}{\lambda_0}}(\int_{x_2}^1dx)^{\frac{\lambda_0-1}{\lambda_0}}
\nnm\\
\leq\
   &\,C_{x_2}E_0^{\frac{1}{2}}
   +C_{x_2}C_{13}^{\frac{1}{\lambda_0}}=:M_{1,b},\label{boundary.c.3.1}
\ema
with
\be
   M_{1,b}
  =:\max\{ C_{x_2}E_0^{\frac{1}{2}}+C_{x_2}C_{13}^{\frac{1}{\lambda_0}},\ M_{1,a}\}.
  \label{boundary.c.3.2}
\ee
From \ef{boundary.c.1.1} for $x\in[x_2,1]$ we have
\bma
    |\rho r^{N-1}u_x|
   =&\big|\frac{1}{\theta}\rho^{\gamma-\theta}-(N-1)ur^{-1}
    +\frac{1}{\theta}(N-1)ur^{-1}
\nnm\\
    &+\frac{1}{\theta}(N-1)\rho^{-\theta}\int_x^1(\frac{u}{r})_y\rho^{\theta}dy
    -\frac{1}{\theta}\rho^{-\theta}\int_x^1r^{1-N}u_{\tau}dy\big|
\nnm\\
   \le
  &\,C+C(\rho^*)^{\gamma-\theta}+C(1-x)^{\frac{1}{2}-\theta\beta}
    +C(1-x)^{\frac{1}{2}(\beta(\theta-1)+1)-\theta\beta}
\nnm\\
  \le
   &\,M_{0,b},\label{boundary.c.3.3}
\ema
with
\be
 M_{0,b}=:\max\{C+C(\rho^*)^{\gamma-\theta},M_{0,a}\},\label{boundary.c.3.4}
\ee
which and \ef{boundary.c.3.1} lead to \ef{boundary.c.3}.

We turn to prove \ef{boundary.c.4}. It is easy to verify
\be
  \rho^{\gamma}\in L^{\infty}(0,T_3,H^1([x_2,1]))
   ,\quad(\rho^{\gamma})_{\tau}\in L^{\infty}(0,T_3,L^2([x_2,1])).
   \label{boundary.c.4.1}
\ee
This implies $\rho^{\gamma}\in C^0([x_2,1]\times[0,T_3])$ and the continuity of density $\rho\in C^0([x_2,1]\times[0,T_3])$. Indeed, it follows from $\ef{3.2}_1$
\be
  (\rho^{\gamma})_{\tau}
  =-\gamma\rho^{\gamma+1}(r^{N-1}u)_x
  =-\gamma\rho^{\gamma+1}r^{N-1}u_x
   -\gamma(N-1)\rho^{\gamma}u r^{-1}
   \in L^{\infty}(0,T_3,L^2([x_2,1])).\label{boundary.c.4.2}
\ee
On the other hand, one derives from \ef{boundary.a.2.6}
\bma
    \int_{x_2}^1(r^{N-1}(\rho^{\gamma})_x)^2dx
    \le
    &\, C\int_{x_2}^1(\rho^{\gamma-\theta}\rho_{0x})^2dx
    +C(\rho^*)^{2(\gamma-\theta)}\int_0^{\tau}\int_{x_2}^1
    (r^{N-1}(\rho^{\gamma})_x)^2dxds
\nnm\\
  &\,+C(\rho^*)^{2(\gamma-\theta)}\int_{x_2}^1(u^2+u_0^2)dx
\nnm\\
  \le
  &\,C(\rho^*)^{2(\gamma-\theta)}+C(\rho^*)^{2(\gamma-\theta)}E_0
  \nnm\\
  &\
   +C(\rho^*)^{2(\gamma-\theta)}
   \int_0^{\tau}\int_{x_2}^1(r^{N-1}(\rho^{\gamma})_x)^2dxds,
   \label{boundary.c.4.3}
\ema
and then
\be
  \int_{x_2}^1(r^{N-1}(\rho^{\gamma})_x)^2dx
   \le
   C(\rho^*)^{2(\gamma-\theta)}(E_0+1)
    (1+C(\rho^*)^{2(\gamma-\theta)}e^{C(\rho^*)^{2(\gamma-\theta)}}),
    \label{boundary.c.4.4}
\ee
which implies
\be
  \int_{x_2}^1[(\rho^{\gamma})_x]^2dx
   \le
   C(\rho^*)^{2(\gamma-\theta)}(E_0+1)
   (1+C(\rho^*)^{2(\gamma-\theta)}e^{C(\rho^*)^{2(\gamma-\theta)}}),
   \label{boundary.c.4.5}
\ee
thus we obtain
\be
\rho^{\gamma}\in L^{\infty}(0,T_3,H^1([x_2,1])),\label{boundary.c.4.6}
\ee
this and \ef{boundary.c.4.2} gives the half of \ef{boundary.c.4}.
We can also show
\be
   u\in L^{\infty}(0,T_3,W^{1,p}([x_2,1])),
  \quad u_{\tau}\in L^{\infty}(0,T_3,L^2([x_2,1])),\label{boundary.c.4.7}
\ee
for any $p\in(1,\beta^{-1})$, so that
\be
  \sup_{\tau\in[0,T_3]}\|u_x\|_{L^p([x_2,1])}
    \le
    \sup_{t\in[0,T_4]}\|(\rho r^{N-1})^{-p}\|_{L^1([x_2,1])}
  \cdot\|\rho r^{N-1}u_x\|_{L^{\infty}([x_2,1]\times[0,T_3])}^p
       \leq C,\label{boundary.c.4.8}
\ee
and
\be
  \|u_{\tau}(\tau)\|_{L^2[x_2,1]}^2
 =\int_{x_2}^1u_{\tau}^2dx
  \le C,\label{boundary.c.4.9}
\ee
this and \ef{boundary.c.4.8} implies the continuity of velocity $u$ on $[x_2,1]\times[0,T_3]$.
\end{proof}


\section{Interior regularities}
It is convenient to make use of \ef{2.1o} directly to investigate the interior regularities of solutions. Indeed, we have

\def\uu{\mathbf{U}}

 \begin{lemma} \label{uniform.m}
Under the assumptions of Theorem~\ref{thm.existence}, there is a
time $T_4\in(0,T_3]$ so that the solution
$(\rho,U)(\mathbf{x},t)=(\rho(r),u(r,t)\frac{\mathbf{x}}{|\mathbf{x}|})$
to the FBVP \ef{2.1o} and \ef{2.1a} satisfies
 \bgr
 \|(\rho,\uu)(t)\|^2_{H^3(\Omega^{in}_t)}
 + \int_0^t  \|(\rho,\nabla\uu)(s)\|^2_{H^3(\Omega^{in}_s)} ds
  \le
  M, \quad t\in[0,T_4],  \label{lm3.9p}
  \egr
where $\Omega^{in}_t=:\{0\le|\mathbf{x}|\le r_{x_2}(t)\}$,
$r_{x_2}(t)$ is the particle path
with $r_{x_2}(0)=r_2\in(r_0,r_1)$, and
$M>0$ is a constant given by \eqref{Mz}. In particular, it holds
\bgr
 \|\nabla\uu(t)\|_{L^\infty(\Omega^{in}_t)}\le C_sM^{1/2},\quad
 \|\nabla\rho(t)\|_{L^\infty((\Omega^{in}_t))}\le M_2,
  \quad t\in[0,T_4],                         \label{lm3.9pd}
 \\
 \frac12\rho_0(r)\le \rho(r,t)\le 2 \rho_0(r),
 \quad  (r,t)\in [0,r_{x_2}(t)]\times[0,T_4],  \label{rhobdda}
\egr
with $C_s>0$ the Sobolev constant for $\|f\|_{L^\infty}\le
C_s\|f\|_{H^2}$.
\end{lemma}
\begin{proof}
We first choose $T_{4,a}\le T_3$ to be small and assume that it holds
 \be
 \|\nabla\uu(t)\|_{L^\infty(\Omega^{in}_t)}\le 2M_0,
 \quad
 \|\nabla\rho(t)\|_{L^\infty(\Omega^{in}_t)}\le 2M_2,
  \quad t\in(0,T_{4,a}],                         \label{rho2}
\ee
with
\be
M_0= C_0M^{1/2}, \label{M_0}
\ee
 and
$M_2=C_0M^{1/2}$. It follows from
\ef{rho2} and \ef{4.9ba}
\be
 \|u(t)\|_{L^\infty([0,r_{x_2}(t)])}
 = \|\mathbf{U}(t)\|_{L^\infty([0,r_{x_2}(t)])}
 \le
 2a_0\|\nabla\mathbf{U}(t)\|_{L^\infty([0,r_{x_2}(t)])}
 =4a_0M_0,
\label{vm_0y}
\ee
which together with \ef{M0,a}, \ef{boundary.c.3}, \ef{rho.+.a} and
\ef{rhoport} (or $\ef{2.1o}_1$) yields
\bgr
  |u(r,t)| \le M_1,
  \quad  (r,t)\in [0,a(t)]\times[0,T_{4,a}],   \label{4.9be}
\\
 \frac12\rho_0(r)\le \rho(r,t)\le 2 \rho_0(r),
 \quad  (r,t)\in [0,a(t)]\times[0,T_{4,a}],  \label{rhobdd}
 \egr
with $M_1=:\max\{M_{1,b},\ 4a_0M_0\}$.

Corresponding to \ef{3.4}, we have the basic energy estimates
\be
 \int_{0}^{a(t)}
  ( \rho \,|\mathbf{U}|^2 + \rho^{\gamma})d\mathbf{x}
  +\int_0^t\int_{0}^{a(t)}
     (|\rho^{\frac{\theta}{2}}\nabla\mathbf{U}|^2
     +|\rho^{\frac{\theta}{2}}{\rm div}\mathbf{U}|^2)d\mathbf{x}ds
      \leq E_0.   \label{E3}
\ee
Take derivative $\partial^\alpha $ with $1\le |\alpha|\le 3$ to
\ef{2.1o} to get
\bgr
  (\partial^\alpha\log\rho)_t + \uu\cdot\nabla\partial^\alpha\log\rho
 +\nabla\cdot \partial^\alpha\uu(\mathbf{x}(t),t) = g_\alpha,    \label{um01}
\\
  \partial^\alpha \uu_t +\uu\cdot\nabla\partial^\alpha\uu
 +\gamma\rho^{\gamma-1}\nabla\partial^\alpha\log\rho
  =   h_\alpha+ f_\alpha +k_\alpha,   \label{um0}
\egr
where
 \bgr
 g_\alpha
 = -\partial^\alpha(\uu\cdot\nabla\log\rho)
     +\uu\cdot\nabla\partial^\alpha\log\rho, \quad
 h_\alpha
 =\partial^\alpha(\rho^{-1}{\rm div} (\rho^\theta\nabla\uu)+(\theta-1)\rho^{-1}\nabla(\rho^\theta{\rm div}\uu)),
\nnm\\
 f_\alpha
 =-(\partial^\alpha(\uu\cdot\nabla\uu)-\uu\cdot\nabla\partial^\alpha\uu),\quad
 k_\alpha
 =-\gamma(\partial^\alpha(\rho^{\gamma-1}\nabla\log\rho
          -\rho^{\gamma-1}\nabla\partial^\alpha\log\rho).  \nnm
\egr
Take inner product of \ef{um0} and $\phi(x(r,t))\partial^\alpha \uu$
over $[0,a(t)]\times[0,t]$, where $\phi=\psi^2(x)$ and $\psi\in C_0^{\infty}([0,1])$ satisfies $0\leq \psi(y) \leq1$, $\psi(y) = 1 $ for $y \in [0,x_2]$, and $\psi(y) = 0$ for $y \in [(1-\eta)x_1,1]$ with $\eta > 0$ small enough so that $(1-\eta)x_1 > x_2$, and use the facts that $\phi_t = \phi_x\rho r^{N-1}u, \phi_r = \phi_x\rho r^{N-1}$, make use
of \ef{um01} and the relation
\[
\begin{aligned}
 \mbox{$
 \int \phi\rho^{\gamma-1}\partial^\alpha\uu\cdot\nabla\partial^\alpha\log\rho$} =&
 \mbox{$\frac12\frac{d}{dt}\int\phi\rho^{\gamma-1}|\partial^\alpha\log\rho|^2
 -\frac12\int [(\phi\rho^{\gamma-1})_t + \nabla\cdot(\phi\rho^{\gamma-1}\uu)]
    |\partial^\alpha\log\rho|^2$}
\nnm\\
 &\mbox{$-\int\phi g_\alpha\rho^{\gamma-1}\partial^\alpha\log\rho
  -\int \partial^\alpha\log\rho\partial^\alpha
   \uu\cdot\nabla(\rho^{\gamma-1}\phi)$},
\end{aligned}
\]
we have
\bma
   &\int_0^{a(t)}\phi(\frac12|\partial^\alpha \uu|^2
   +\frac{\gamma}{2}\rho^{\gamma-1}
   |\partial^\alpha\log\rho|^2)r^{N-1}dr
\nnm\\
    =&\int_0^{a(0)}\phi(\frac12|\partial^\alpha\uu|^2
      +\frac{\gamma}{2}\rho^{\gamma-1}
       |\partial^\alpha\log\rho|^2)r^{N-1}(r,0)dr
\nnm\\
        &\, -\frac12\int_0^t\int_0^{a(s)}\phi_s
         |\partial^\alpha\uu|^2r^{N-1}drds
         -\int_0^t\int_0^{a(s)}\phi\partial^\alpha\uu
     \cdot(\uu\cdot\nabla\partial^\alpha\uu)r^{N-1}drds
\nnm\\
    &\,+\frac{\gamma}{2}\int_0^t\int_0^{a(s)}[(\phi\rho^{\gamma-1})_s
       +\nabla\cdot(\phi\rho^{\gamma-1}\uu)]
        |\partial^\alpha\log\rho|^2r^{N-1}drds
\nnm\\
       &\, +\gamma\int_0^t\int_0^{a(s)}\partial^\alpha
        \log\rho\partial^\alpha\uu\cdot\nabla
        (\phi\rho^{\gamma-1})r^{N-1}drds
\nnm\\
    &\,+\int_0^t\int_0^{a(s)}\phi\partial^\alpha\uu
      (h_\alpha+f_\alpha+k_\alpha)r^{N-1}drds
\nnm\\
    &\,+\gamma\int_0^t\int_0^{a(s)}\phi
       r^{N-1} g_\alpha\rho^{\gamma-1}
        \partial^\alpha\log\rho drds, \nnm
\ema
after a direct computation that
\bma
   &\int_0^{a(t)}\phi(|\partial^\alpha\uu|^2
    +\rho^{\gamma-1}|\partial^\alpha\log\rho|^2)r^{N-1}dr
\nnm\\
     &\,+\int_0^t\int_0^{a(s)}\phi\rho^{\theta-1}
      (|\nabla\partial^\alpha\uu|^2
       +|\nabla\cdot\partial^\alpha\uu|^2)r^{N-1}drds
\nnm\\
  \leq
    &C\|(\uu_0,\rho_0)\|^2_{H^3([0,r_1])}
     +C_{x_0,x_1}(1+M_0+M_1+M_2+(\rho^*)^{\gamma-1})
      (\delta_4^2+\delta_5^2+\delta_6^2)t
\nnm\\
   &\,+C_{x_0,x_1}(M_0^2+M_1^2+M_2^2+\delta_4^2+\delta_5^2)
      \int_0^t\int_0^{a(s)}\phi(|\partial^\alpha \uu|^2
       +\rho^{\gamma-1}|\partial^\alpha\log\rho|^2)r^{N-1}drds
\nnm\\
   &\,+C(1+M_0+M_2+\delta_4+\delta_5)\int_0^t
     (\sum_{|\alpha|=1}^3\int_0^{a(s)}\phi (|\partial^\alpha\uu|^2
     +\rho^{\gamma-1}|\partial^\alpha\log\rho|^2)r^{N-1}dr)^2ds, \nnm
\ema
which implies
\bma
 & Y(t)
  +\sum_{|\alpha|=1}^3\int_0^t\int_0^{a(s)}\phi\rho^{\theta-1}
      (|\nabla\partial^\alpha\uu|^2
       +|\nabla\cdot\partial^\alpha\uu|^2)r^{N-1}drds
       \nnm\\
\le&
   K_6+K_7t+K_8\int_0^tY^2(s)ds, \label{um08}
\ema
where
\be
 K_6= C\|(\uu_0,\rho_0)\|^2_{H^3([0,r_1])},\nnm
\ee
\be
 K_8= C_{x_0,x_1}(M_0^2+M_1^2+M_2^2+\delta_4^2+\delta_5^2)
      +C(1+M_0+M_2+\delta_4+\delta_5), \nnm
\ee
\be
 K_7= C_{x_0,x_1}(1+M_0+M_1+M_2+(\rho^*)^{\gamma-1})
      (\delta_4^2+\delta_5^2+\delta_6^2)
      +C_{x_0,x_1}(M_0^2+M_1^2+M_2^2+\delta_4^2+\delta_5^2), \nnm
\ee
$$
 T_{4,b}=:\min\{K_7^{-1}C\|(\uu_0,\rho_0)\|^2_{H^3([0,r_1])},\ T_{4,a}\, \},\nnm
$$
 and
\[
 Y(t)=\int_0^{a(t)}\phi
    ( |(\partial \uu,\partial^2 \uu,\partial^3 \uu)|^2
      +
      |(\partial\log\rho,\partial^2\log\rho,\partial^3\log\rho|^2)(r,t)r^{N-1}dr.
\]
We apply Gronwall's inequality to have
\bma
 & \sum_{|\alpha|=1}^3\int_0^{a(t)}\phi
    ( |\partial^\alpha \uu|^2 +|\partial^\alpha\log\rho|^2 )r^{N-1}dr
 \nnm\\
   &\,+\sum_{|\alpha|=1}^3\int_0^t\int_0^{a(s)}\phi\rho^{\theta-1}
      (|\nabla\partial^\alpha\uu|^2
       +|\nabla\cdot\partial^\alpha\uu|^2)r^{N-1}drds
\nnm\\
  \le\,
  &  4C\|(\uu_0,\rho_0)\|^2_{H^3([0,r_1])},\quad   t\in[0,T_4],   \label{um10}
\ema
for $T_4=:\min\{\, T_{4, b},\, T_{4,c}\,\}$ with
$T_{4,c}=(4K_6K_8)^{-1}$. The \ef{um10} together with \ef{rhobdd}, Lemma~\ref{boundary.a} and
$$
  \|\rho\|_{H^3(\Omega_t^{in})}^2
\le
  C_{x_0,x_1}(\|\log\rho\|_{H^3(\Omega_t^{in})}^2)^3,
$$ leads to \ef{lm3.9p} with $M$ given by
\be
 M=C_{x_0,x_1}(4C\|(\uu_0,\rho_0)\|^2_{H^3(\Omega^{in}_+)}
   )^3.    \label{Mz}
\ee
The estimates \ef{rho2} and \ef{lm3.9pd} follow respectively from
\ef{lm3.9p} and the Sobolev embedding theorem, and \ef{rhobdda}
follows from \ef{rhoport} and \ef{lm3.9p}.
\end{proof}

\section{Proof of the main results}
\begin{proposition}[\textbf{Existence and Uniqueness}]
\label{approx.euler} Under the assumptions of
Theorem~\ref{thm.existence}, there exists a time $T_*>0$ dependent
of initial data, so that the FBVP \eqref{2.1o} and \eqref{2.1a}
admits a unique solution
$$
\mbox{$(\rho,\rho\mathbf{U},a)(\mathbf{x},t) =
 (\rho(r,t),\rho u(r,t)\frac{\mathbf{x}}{r},a(t))$},\
  r=|\mathbf{x}|,\  (r,t)\in [0,a(t)]\times[0,T_*],
 $$
which satisfies
\bgr
 \|\nabla \mathbf{U}\|_{L^\infty([0,r_{x_2}(t)]\times[0,T_*])}\le M_0,
 \quad
 \|\rho r^{N-1}u_x\|_{L^\infty([x_0,1]\times[0,T_*])}\le M_0, \label{vm_0a}
\\
 \mbox{$\frac12$}\rho_0(x)\le \rho(x,\tau)\le 2\rho_0(x),
    \quad (x,\tau)\in[0,1]\times[0,T_*],     \label{rho_betab}
\egr
for constant $M_0>0$ dependent of initial data, $\rho_+=2\rho^*$ and $\rho_-=\frac12\rho_*$. In addition,
$(\rho,\rho\mathbf{U},a)$ satisfies the
inner regularities in Euler coordinates
\bgr
 \int_{0}^{a(t)}
  ( \rho \,|\mathbf{U}|^2 + \rho^{\gamma})d\mathbf{x}
  +\int_0^t\int_{0}^{a(t)}
     \rho^{\theta}|\nabla\mathbf{U}|^2d\mathbf{x}ds
     \leq C,
   \ \ t \in [0,T_*],             \label{pro3.1a}
\\
 \|(\rho,\mathbf{U})(t)\|^2_{H^3(\Omega^{in}_t)}
 + \sum_{|\alpha|=1}^3
   \int_0^t\int_{\Omega^{in}_s}
  (  |\partial^{\alpha} \rho |^2
    +|\partial^{\alpha+1} \mathbf{U}|^2)d\mathbf{x}ds
  \le C_{in},
   \ \ t \in [0,T_*],    \label{pro3.1b}
  \egr
where $\Omega^{in}_t=\{0\le |\mathbf{x}|\le r_{x_1}(t)\}$,
$r_{x_1}(t)$ is defined by $r'_{x_1}(t)=u(r_{x_1}(t),t)$ with
$r_{x_1}(0)=r_1\in(r_2,a_0)$ and $x_1=1- \int_{r_1}^{a_0}\rho_0
r^{N-1}dr\in(x_2,1)$, and $C>0$ and $C_{in}>0$ are constants, and the
boundary regularities in Lagrange coordinates
\bgr
   \int_{x_2}^1u^{2k}dx+\int_0^{\tau}\int_{x_2}^1\rho^{\theta+1}u^{2k-2}r^{2N-2}u_x^2
   dx ds
   \leq C_{b},
\ \ \tau \in [0,T_*],  \label{pro3.1c}
  \\
   \int_{x_2}^1u_{\tau}^2dx+\int_0^{\tau}\int_{x_2}^1(\rho^{\theta+1}r^{2N-2}u_{xs}^2
  + \rho^{\theta-1}r^{-2}u_{s}^2)dx ds
   \leq C_{b},
\ \ \tau \in [0,T_*],  \label{pro3.1d}
\egr
with the integer $1 \leq k \leq 2m$, and $C_{b}>0$ is a constant.
\end{proposition}
\bigskip

\underline{\it Proof of the Theorem~\ref{thm.existence}}. With the
estimates we have obtained in Sections 3-5, we can apply the method
of difference scheme and compactness arguments as in [3, 8] and
references therein, to prove the existence of weak solutions to the
FBVP \eqref{2.1o} and \eqref{2.1a}, we omit here. Next, we apply the
idea in [3] to prove the uniqueness. Let $(\rho_1,u_1,r_1)$ and
$(\rho_2,u_2,r_2)$ are two solutions to the FBVP
~\eqref{3.2}--\eqref{3.3b}, and denote
$$
(\varrho,\omega,R)(x,\tau)=(\rho_1-\rho_2,u_1-u_2,\frac{r_1}{r_2}-1)(x,\tau).
$$
Based on Proposition~\ref{approx.euler}, we can derive the
following estimates
 \bgr
 0< c\rho_*(\rho^*)^{-1}
 \le \frac{\rho_1(x,\tau)}{\rho_2(x,\tau)}
    + \frac{\rho_2(x,\tau)}{\rho_1(x,\tau)}
     \le C\rho^*(\rho_*)^{-1},  \ \ (x,\tau)\in[0,1]\times[0,T_*],
 \\
 \frac{|u_1|}{r_1}+\frac{|u_2|}{r_2}
  +| \rho_1^{1+\theta}r_1^{N-1}u_{1x}| + |\rho_2^{1+\theta}r_2^{N-1}u_{2x}|\le C,\ \ x\in[0,1],
 \\
 0< C_{x_0}^{-1}(2a_0)^{-1}
  \le
     \frac{r_1(x,\tau)}{r_2(x,\tau)}
   + \frac{r_2(x,\tau)}{r_1(x,\tau)}
  \le
   2C_{x_0}a_0,\ (x,\tau)\in[x_0,1]\times[0,T_*],
 \\
 (\rho_*(\rho^*)^{-1})^{1/N}
  \le  \frac{r_1(x,\tau)}{r_2(x,\tau)}+ \frac{r_2(x,\tau)}{r_1(x,\tau)}
  \le (\rho^*(\rho_*)^{-1})^{1/N},\   (x,\tau)\in[0,x_0]\times[0,T_*],
 \egr
with which, we can show the uniqueness of the solutions. Indeed,
From \ef{3.2} and using Young's inequality, we have
\bma
  \frac{d}{d\tau} \int_0^1 \rho_1^{\theta-1}R^2dx
 =&\int_0^1 (2\rho_1^{\theta-1} R R_\tau
  +(\theta-1)\rho_1^{\theta-2}\rho_{1\tau} R^2)dx
\nnm\\
 =&(1-\theta)\int_0^1\rho_1^\theta(r_1^{N-1}u_1)_xdx
  +2\int_0^1\rho_1^{\theta-1}R(\frac{u_1}{r_2}
  -\frac{r_1u_2}{r_2^2})dx
 \nnm\\
 \le
 &
  \varepsilon\int_0^1 \rho_1^{\theta-1}\frac{\omega^2}{r_1^2}dx
  +C_\varepsilon\int_0^1\rho_1^{\theta-1}R^2dx,\label{7.a}
\ema
and
\bma
  \frac{d}{d\tau} \int_0^1 \rho_1^{\theta-3}\varrho^2dx
  =& 2\int_0^1\rho_1^{\theta-3} \varrho (\rho_{1\tau}-\rho_{2\tau})dx
   +(\theta-3)\int_0^1\rho_1^{\theta-4}\rho_{1\tau} \varrho^2dx
\nnm\\
  =&(3-\theta)\int_0^1 \rho_1^{\theta-2} (r_1^{N-1}u_1)_x \varrho^2dx
   +2\int_0^1 \rho_1^{\theta-3} \varrho
   (\rho_1^2(r_1^{N-1}u_1)_x-\rho_2^2(r_2^{N-1}u_2)_x)dx
\nnm\\
  \le
   &
   \varepsilon\int_0^1 \rho_1^{\theta-1}\frac{\omega^2}{r_1^2}dx
   +\varepsilon\int_0^1 \rho_1^{\theta+1}r_1^{2(N-1)}\omega_x^2
   +C_\varepsilon\int_0^1 \rho_1^{\theta-3}\varrho^2
   +C_\varepsilon\int_0^1\rho_1^{\theta-1}R^2dx,\label{7.b}
\ema
where $\varepsilon>$ is chosen later and $C_{\varepsilon}>0$ a constant.

From the equation $\ef{3.2}_2$ and boundary condition, we get
\bma
  \frac{d}{d\tau} \int_0^1 \frac12 \omega^2dx
   =&\int_0^1 \{-\theta \rho_1^{\theta+1}(r_1^{N-1}u_1)_x (r_1^{N-1}\omega)_x
    -\theta \rho_2^{\theta+1}(r_2^{N-1}u_2)_x (r_2^{N-1}\omega)_x\}dx
\nnm\\
    &+\int_0^1 \{\rho_1^\gamma(r_1^{N-1}\omega)_x-\rho_2^\gamma(r_2^{N-1}\omega)_x\}dx
\nnm\\
    &+(N-1)\int_0^1 \{\rho_1^\theta(r_1^{N-2}u_1\omega)_x-\rho_2^\theta(r_2^{N-2}u_2\omega)_x\}dx.
\nnm
\ema
Using the similar argument as that in Lemma~\ref{lm3.2} and Cauchy-Schwartz inequality, we have
\bma
\frac{d}{d\tau} \int_0^1 \frac12 \omega^2dx
   \le&
   \varepsilon\int_0^1\rho_1^{\theta+1}r_1^{2(N-1)}\omega_x^2dx
   +\varepsilon\int_0^1 \rho_1^{\theta-1}\frac{\omega^2}{r_1^2}dx
 \nnm\\
  &+C_\varepsilon\int_0^1\rho_1^{\theta-3} \varrho^2dx
  +C_\varepsilon\int_0^1 \rho_1^{\theta-1}R^2dx,
  \nnm\\
  &\,
   -\frac12\int_0^1\{(\theta-1+\mbox{$\frac1N$)}\rho_1^{\theta+1}[(r_1^{N-1}\omega)_x]^2
\nnm\\
   &+\mbox{$\frac{N-1}{N}$}\rho_1^{\theta+1}[r_1^{N-1}\omega_x
   -\frac{\omega}{r_1\rho_1}]^2\}dx,
\label{7.c}
\ema
where $\varepsilon>0$ small enough and $C_{\varepsilon}>0$ is a constant.

Apply the Gronwall's inequality to the summation
of \ef{7.a}--\ef{7.c}, we can finally obtain
\bma
 &\int_0^1( w^2 + \rho_1^{\theta-1}R^2  + \rho_1^{\theta-3}\varrho^2 )(x,\tau) dx
 \nnm\\
 \le&
 C \int_0^1 ( w^2 +\rho_1^{\theta-1}R^2  + \rho_1^{\theta-3}\varrho^2)(x,0) dx=0, \quad(x,\tau)\in[0,1]\times[0,T_*],\nnm
\ema
which implies $(\rho_1,u_1,r_1)=(\rho_2,u_2,r_2)$.
\bigskip\\
{\bf Acknowledgments}\  \  The author would like to thank the referee for the helpful comments and suggestions on the paper.   \par
The author is grateful to Professor Hai-Liang Li, his supervisor,
on the discussion and suggestions about the problem. The research of
J.Liu is partially supported by the NNSFC No. 10871134 and  the
AHRDIHL Project of Beijing Municipality No.PHR201006107.

\bibliographystyle{unsrtnat}

%
\end{document}